\documentclass[reqno]{amsart}
\usepackage{amsmath}
\usepackage[active]{srcltx}
\usepackage{amssymb}
\usepackage[dvips]{graphicx}
\usepackage{color}
\newtheorem{Proposition}{Proposition}[section]
\newtheorem{Corollary}{Corollary}[section]
\newtheorem{Lemma}{Lemma}[section]
\newtheorem{Theorem}{Theorem}[section]
 \newtheorem*{theorem*}{Theorem}

\theoremstyle{definition}
\newtheorem{Definition}{Definition}[section]

\theoremstyle{definition}
\newtheorem{Remark}{Remark}[section]
\numberwithin{equation}{section}
\usepackage{enumerate}

\begin{document}

\title{Uniform distribution of subpolynomial functions along primes and applications} 
\author[V. Bergelson]{Vitaly Bergelson}
\thanks{The first author gratefully acknowledges the support of the NSF under grant DMS-1162073.}
\address[V. Bergelson]{Department of Mathematics\\ Ohio State University \\ Columbus, OH 43210, USA}
\email{vitaly@math.ohio-state.edu}

\author[G. Kolesnik]{Grigori Kolesnik}
\address[G. Kolesnik]{Department of Mathematics\\ California State University \\ Los Angeles, CA 90032, USA}
\email{gkolesnik@sbcglobal.net}

\author[Y. Son]{Younghwan Son}
\address[Y. Son]{Faculty of Mathematics and Computer Science\\ Weizmann Institute of Science \\ Rehovot, 7610001,Israel}
\email{younghwan.son@weizmann.ac.il}

%\date{\today}
\bigskip

\gdef\shorttitle{Uniform distribution of subpolynomial functions along primes}
\maketitle

\setcounter{section}{0}

\begin{abstract} 
Let $H$ be a Hardy field (a field consisting of germs of real-valued functions at infinity that is closed under differentiation) and let $f \in H$ be a subpolynomial function. 
Let $\mathcal{P} = \{2, 3, 5, 7, \dots \}$ be the (naturally ordered) set of primes. 
We show that $(f(n))_{n \in \mathbb{N}}$ is uniformly distributed $\bmod \, 1$ if and only if $(f(p))_{p \in \mathcal{P}}$ is uniformly distributed $\bmod \, 1$. 
This result is then utilized to derive various ergodic and combinatorial statements which significantly generalize the results obtained in \cite{BKMST}.
\end{abstract}

%{\center{\today}}

\section{Introduction}
In the recent paper \cite{BKMST} new results on sets of recurrence involving the prime numbers were established, 
thereby unifying and refining some previous results obtained in \cite{Sa1}, \cite{Sa2}, \cite{Sa3}, \cite{F}, \cite{KM} and \cite{BL}.

Here are the formulations of some of the results from \cite{BKMST} that are relevant to the discussion in this introduction.
\begin{Theorem} [Theorem 3.1 in \cite{BKMST}]
\label{thm0.1}
Let $c_1, c_2, \dots , c_k$ be distinct positive real numbers such that $c_i \notin \mathbb{N} (:= \{1, 2, 3, \dots \})$ for $i=1,2, \dots, k$. 
Let $U_1, \dots , U_k$ be commuting unitary operators on a Hilbert space $\mathcal{H}$.  
Then,
$$\lim_{N \rightarrow \infty} \frac{1}{N}  \sum_{n=1}^N U_1^{[p_n^{c_1}]} \cdots U_k^{[p_n^{c_k}]} f = f^*,$$
strongly in $\mathcal{H}$, where $p_n$ denotes the $n$-th prime and $f^*$ is the projection of $f$ on $\mathcal{H}_{inv} := \{ f \in \mathcal{H} :  U_i f = f \,\, \textrm{for all} \,\, i\}$.
\end{Theorem}

\begin{Theorem}[Corollary 3.1 in \cite{BKMST}]
\label{thm0.2}
Let $c_1, c_2, \dots , c_k$ be positive non-integers. 
Let $T_1, T_2, \dots , T_k$ be commuting, invertible measure preserving transformations on a probability space $(X, \mathcal{B}, \mu)$. 
Then, for any $A \in \mathcal{B}$ with $\mu(A) > 0$, one has 
$$\lim_{N \rightarrow \infty} \frac{1}{N} \sum_{n=1}^{ N} \mu(A \cap T_1^{-[p_{n}^{c_1}]} \cdots T_k^{-[p_n^{c_k}]} A) \geq \mu^2(A),$$
where $p_n$ denotes the $n$-th prime.
\end{Theorem}

\noindent Recall that the upper Banach density of a set $E \subset \mathbb{Z}^k$ is defined to be
$$d^*(E) = \sup_{\{\Pi_n\}_{n \in \mathbb{N}}} \limsup_{n \rightarrow \infty} \frac{|E \cap \Pi_n|}{| \Pi_n|},$$
where the supremum is taken over all sequences of parallelepipeds
$$ \Pi_n = [a_n^{(1)}, b_n^{(1)}] \times \cdots \times [a_n^{(k)}, b_n^{(k)} ] \subset \mathbb{Z}^k, \,\, n \in \mathbb{N} \,\,
\text{with} \, \, \, b_n^{(i)} - a_n^{(i)} \rightarrow \infty \,\,(1 \leq i \leq k).$$
\begin{Theorem}[Corollary 3.2 in \cite{BKMST}]
\label{thm1.3}
Let $c_1, \dots ,c_k$ be positive non-integers.
If $E \subset \mathbb{Z}^k$ with ${d^*}(E) > 0$, then there exists a prime $p$ such that $([p^{c_1}], \dots , [p^{c_k}] ) \in E - E$. 
Moreover,
$$\liminf_{N \rightarrow \infty} \frac{ | \{ p \leq N:   ([p^{c_1}], \dots , [p^{c_k}] ) \in E - E\} | }{ \pi(N) } \geq {d^*}(E)^2.$$
\end{Theorem}
 
To obtain combinatorial corollaries, such as Theorem \ref{thm1.3}, from the ergodic statements 
one utilizes the $\mathbb{Z}^k$-version of Furstenberg's correspondence principle (see, for example, Proposition 7.2 in \cite{BMc}). 
\begin{Proposition}
\label{correspondence}
Given $E \subset \mathbb{Z}^k$ with $d^*(E)>0$, 
there is a probability space $(X, \mathcal{B} , \mu)$, 
commuting invertible measure preserving transformations $T_1, T_2, \dots , T_k$ of $X$ 
and $A \in \mathcal{B} $ with $d^*(E) = \mu(A)$ such that for any $m \in \mathbb{N}$ and $\bold{n}_1, \bold{n}_2, \dots , \bold{n}_m \in \mathbb{Z}^k $ one has
$$d^*(E \cap (E- \bold{n}_1) \cap (E- \bold{n}_2) \cap \cdots \cap (E - \bold{n}_m)) \geq \mu(A \cap T^{- \bold{n}_1}A \cap \cdots \cap T^{-\bold{n}_m} A),$$
where for $\bold{n} = (n_1, \dots , n_k)$, $T^{\bold{n}} = T_1^{n_1} \cdots T_k^{n_k}.$
\end{Proposition}

It follows from the result obtained in \cite{BKMST} that both Theorems \ref{thm0.1} and \ref{thm0.2} remain true 
if one replaces $[p_n^{c_i}], i=1, 2, \dots, k,$ by $[(p_n-h)^{c_i}]$ for any fixed $h \in \mathbb{Z}$. 
Now, when $h = \pm 1$, these results hold (in a slightly modified form) for $c_i \in \mathbb{N}$ as well 
and one has the following theorem which provides a simultaneous extension of various classical results.
$\mathcal{P}$ denotes the set of prime numbers $\{2, 3, 5, 7, \dots \}$.
\begin{Theorem}[Theorem 5.1 and Corollary 5.1 in \cite{BKMST}]
\label{corB}
Let $$D_h = \{ \left((p-h)^{\alpha_1}, \dots, (p-h)^{\alpha_k}, [(p-h)^{\beta_1}], \dots, [(p-h)^{\beta_l}]\right) | p \in \mathcal{P}\},$$
where $\alpha_1, \dots, \alpha_k$ are positive integers and $\beta_1, \dots, \beta_l$ are positive non-integers.
\begin{enumerate}[(i)]
\item For any measure preserving $\mathbb{Z}^{k+l}$-action $(T^{ \bold{d}})_{\bold{d} \in \mathbb{Z}^{k+l}}$ on a probability space $(X, \mathcal{B}, \mu)$,
$\{\bold{d} \in D_h: \mu (A \cap T^{- \bold{d}}A) > \mu^2(A) - \epsilon\}$ has positive lower relative density in $D_h$ for $h= \pm 1$.
\footnote{ For sets $A \subset B \subset \mathbb{Z}^m$, the relative density and the lower relative density of $A$ with respect to $B$ are defined as 
$$ \lim_{n \rightarrow \infty} \frac{|A \cap [-n,n]^m|}{|B \cap [-n,n]^m|} \quad \text{and} \quad  \liminf_{n \rightarrow \infty} \frac{|A \cap [-n,n]^m|}{|B \cap [-n,n]^m|}.$$ }
\item If $E \subset \mathbb{Z}^{k+l}$ with ${d^*}(E) > 0$, then for any $\epsilon > 0$
\begin{equation*}
 \{ \bold{d} \in D_h : d^*(E \cap E - \bold{d} ) \geq d^*(E)^2 - \epsilon \} 
 \end{equation*}
 has positive lower relative density in $D_h$ for $h=\pm 1$.
 Furthermore,
 $$\liminf_{N \rightarrow \infty} \frac{\left| \{ p \leq N :\left(  (p - 1)^{\alpha_1}, \cdots , (p - 1)^{\alpha_k}, [(p - 1)^{\beta_1}], \cdots , [(p - 1)^{\beta_l}]  \right) \in E - E \} \right| }{\pi(N)} > 0.$$
  $$\liminf_{N \rightarrow \infty} \frac{\left| \{ p \leq N :\left(  (p + 1)^{\alpha_1}, \cdots , (p + 1)^{\alpha_k}, [(p + 1)^{\beta_1}], \cdots , [(p + 1)^{\beta_l}]  \right) \in E - E \} \right| }{\pi(N)} > 0.$$
\end{enumerate} 
\end{Theorem}

Theorems \ref{thm0.1} and \ref{corB} are derived in \cite{BKMST}  with the help of the following equidistribution result. 
\begin{Theorem}[Theorem 2.1 in \cite{BKMST}]
\label{thm0.3}
Let $\xi(x) = \sum_{j=1}^{m} \alpha_j x^{\theta_j}$, 
where $0 <\theta_1 < \theta_2 < \cdots < \theta_m$, $\alpha_j$ are non-zero reals, 
and assume that if all $\theta_j \in \mathbb{N}$, then at least one $\alpha_j$ is irrational. 
Then the sequence $ ( \xi(p) )_{p \in \mathcal{P}} $ is u.d. mod $1$.\footnote{We are tacitly assuming that the set $\mathcal{P} = (p_n)_{n \in \mathbb{N}}$ is naturally ordered, so that $(f(p))_{p \in \mathcal{P}}$ is just another way of writing $(f(p_n))_{n \in \mathbb{N}}$.} 
\end{Theorem}

Note that in the case when all $\theta_j \in \mathbb{N}$, Theorem \ref{thm0.3} reduces to the classical result of Rhin \cite{Rh} which states that if $f(x)$ is a polynomial with at least one coefficient other than the constant term irrational, then $(f(p))_{p \in \mathcal{P}}$ is uniformly distributed $\bmod \, 1$. Incidentally, in Theorem \ref{thm4} below we provide a new short proof of (a slight extension of) Rhin's theorem. While Theorem \ref{thm0.3} forms a rather natural extension of Rhin's result, one would like to know whether the phenomenon of uniform distribution along primes holds for more general regularly behaving and, say, eventually monotone functions. Besides being of independent interest, any such extension of Theorem \ref{thm0.3} allows one to obtain new applications to ergodic theory and combinatorics. In this context it is natural to consider functions belonging to Hardy fields.

Let $B$ denote the set of germs at $+ \infty$ of continuous real functions on $\mathbb{R}$.
Note that $B$ forms a ring with respect to pointwise addition and multiplication.
\begin{Definition} 
A {\em Hardy field} is any subfield of $B$ which is closed under differentiation. By ${\bf U} \subset B$ we denote the union of all Hardy fields.
\end{Definition}

A classical example of a Hardy field is provided by field $L$ of logarithmico-exponential functions introduced in \cite{Har1, Har2}, that is, the collection of all functions that can be constructed using the real 
constants, the functions $e^x$ and $\log x$ and the operations of addition, multiplication, division and composition of functions. 
%if the functions are well defined for large $x$.

For any $f \in {\bf U}$, $\lim\limits_{x \rightarrow \infty} f(x)$ exists as an element of $\mathbb{R} \cup \{- \infty, \infty\}$. This implies that periodic functions such as $\sin x$ and $\cos x$ do not belong to ${\bf U}$. Also if $f_1$ and $f_2$ belong to the same Hardy field, then the limit $\lim\limits_{x \rightarrow \infty} \frac{f_1(x)}{f_2(x)}$ exists (it may be infinite). See \cite{Bos} and some references therein for more information about Hardy fields.

A function $f \in \bold{U}$ is said to be subpolynomial if, for some $n \in \mathbb{N}$, $|f(x)| < x^n$ for all large enough $x$. It was proved in \cite{Bos} that if $f \in {\bf U}$ is a subpolynomial function, then $(f(n))_{n \in \mathbb{N}}$ is uniformly distributed $\bmod \, 1$ if and only if for any $P(x) \in \mathbb{Q}[x]$ one has $\lim\limits_{x \rightarrow \infty} \frac{f(x) - P(x)}{\log x} = \pm \infty$. One of the main results of this paper states that an equidistribution result similar to Theorem \ref{thm0.3} holds for any subpolynomial function satisfying Boshernitzan's condition.

\begin{Theorem}[Theorem \ref{main} in Section 3]
\label{main-intro}
For a subpolynomial function $f(x) \in {\bf U}$, the following conditions are equivalent:
\begin{enumerate}
\item $(f(n))_{n \in \mathbb{N}}$ is u.d. $\bmod \, 1$.
\item $(f(p))_{p \in \mathcal{P}}$ is u.d. $\bmod \, 1$.
\item For any $P \in \mathbb{Q}[x]$,
\begin{equation*} 
\lim_{x \rightarrow \infty} \frac{f(x) - P(x)}{\log x} = \pm \infty.
\end{equation*}
\end{enumerate}
\end{Theorem}

We will give now a sample of various applications of Theorem \ref{main-intro} obtained in this paper. One of these applications is an extension of the above Theorem \ref{thm0.1}. For a given Hardy field $H$, let ${\bf H}$ be the set of all subpolynomial functions $\xi \in H$ such that 
$$ \text{either} \,\,\, \lim\limits_{x \rightarrow \infty} \frac{\xi(x)}{x^{l+1}} = \lim\limits_{x \rightarrow \infty} \frac{x^l}{\xi(x)} = 0 \,\, \text{ for some} \,\, l \in \mathbb{N}, 
\,\,\, \text{or} \,\,\, \lim\limits_{x \rightarrow \infty} \frac{\xi(x)}{x} = \lim\limits_{x \rightarrow \infty} \frac{\log x}{\xi(x)} =0.$$

\begin{Theorem}[Theorem \ref{ergodic} in Section 4]
%\label{ergodic}
Let $U_1, \dots, U_k$ be commuting unitary operators on a Hilbert space $\mathcal{H}$. 
Let $\xi_1, \dots , \xi_k \in {\bf H}$ such that $\sum_{i=1}^k b_i \xi_i \in {\bf H}$ for any $(b_1, \dots, b_k) \in \mathbb{R}^k \backslash \{(0, 0, \dots, 0)\}$.
Then,
\begin{equation*}
%\label{ergodiceq}
\lim_{N \rightarrow \infty} \frac{1}{N}  \sum_{n=1}^N U_1^{[\xi_1(p_n)]} \cdots U_k^{[\xi_k(p_n)]} f = f^*,
\end{equation*}
strongly in $\mathcal{H}$, where $f^*$ is the projection of $f$ on $\mathcal{H}_{inv} (:= \{ f \in \mathcal{H} :  U_i f = f \,\, \textrm{for all} \,\, i\})$.
\end{Theorem}

\begin{Theorem}[cf.\ Theorem \ref{A} in Section 4]
Let $\xi_1, \dots , \xi_k \in {\bf H}$ such that $\sum\limits_{i=1}^k b_i \xi_i \in {\bf H}$ for any $(b_1, \dots, b_k) \in \mathbb{R}^k \backslash \{(0, 0, \dots, 0)\}$. 
Let $L: \mathbb{Z}^k \rightarrow \mathbb{Z}^m$ be a linear transformation and $(\psi_1(n), \dots, \psi_m(n)) = L([\xi_1(n)], \dots, [\xi_k(n)])$. 
Let $T_1, T_2, \dots , T_m$ be commuting, invertible measure preserving transformations on a probability space $(X, \mathcal{B}, \mu)$. 
Then, for any $A \in \mathcal{B}$ with $\mu(A) > 0$, one has 
$$\lim_{N \rightarrow \infty} \frac{1}{N} \sum_{n=1}^{ N} \mu(A \cap T_1^{-\psi_1(p_{n})} \cdots T_m^{-\psi_m(p_n)} A) \geq \mu^2(A).$$
\end{Theorem}
 
\noindent Denote
\begin{equation*}
%\label{Dset}
\begin{split}
 \bold{D}_{-1} &= \{ \left(P_1(p-1), \dots , P_l (p-1), [\xi_1(p)], \dots , [\xi_k (p)]  \right) | \, p \in\mathcal{P} \}, \\
 \bold{D}_{1} &= \{ \left( P_1( p+1), \dots ,P_l (p+1), [\xi_1(p)], \dots , [\xi_k (p)]  \right) | \, p \in\mathcal{P} \},
\end{split}
\end{equation*}
where $P_1, \dots,P_l \in \mathbb{Z}[x]$ with $P_i(0)=0$ for all $1 \leq i \leq l$ and $\xi_1, \dots, \xi_k \in {\bf H}$ such that $\sum_{j=1}^k b_j \xi_j(x) \in {\bf H}$ for any $(b_1, \dots, b_k) \in \mathbb{Z}^k \backslash \{(0, 0, \dots, 0)\}$. 
%Let $L: \mathbb{Z}^{l+k} \rightarrow \mathbb{Z}^{l+k}$ be a non-zero linear map and $\bold{D}_i = L(D_i)$.   

\begin{Theorem}[cf.\ Theorem \ref{recurrence}  and Remark \ref{end} in Section 4]
Enumerate the elements of ${\bold{D}}_i$, $(i = \pm 1)$, as follows:
$$\bold{d}_{n,i} =  \left( P_1(p_n + i), \dots , P_l(p_n + i) , [\xi_1(p_n )], \dots , [\xi_k(p_n )]  \right) \quad n=1, 2, \dots .$$
Let $(T^{\bold{d}})_{\bold{d} \in \mathbb{Z}^{l+k}}$ be a measure preserving $\mathbb{Z}^{m}$-action on a probability space $(X, \mathcal{B}, \mu)$.
Then ${D}_1$ and ${D}_{-1}$ are ``averaging" sets of recurrence:
  \begin{equation*}
  \lim_{N \rightarrow \infty} \frac{1}{N} \sum_{n=1}^N \mu(A \cap T^{-\bold{d}_{n,i}} A) > 0 \quad (i = \pm 1).
  \end{equation*} 
Moreover, for any $\epsilon>0$, $\{ \bold{d} \in \bold{D}_i : \mu(A \cap T^{-\bold{d}}A ) \geq \mu^2(A) - \epsilon \}$ has positive lower relative density in $\bold{D}_i$ $(i = \pm 1)$. 
\end{Theorem}

\begin{Theorem}[cf. Corollary \ref{semi ergodic cor} and Remark \ref{end} in Section 4]
If $E \subset \mathbb{Z}^{l+k}$ with ${d^*}(E) > 0$, then for any $\epsilon > 0$,
\begin{equation*}
\{ \bold{d} \in \bold{D}_i : d^*(E \cap E - \bold{d} ) \geq d^*(E)^2 - \epsilon \} 
 \end{equation*}
has positive lower relative density in $\bold{D}_i$ for $i= \pm 1$.
Furthermore,
$$\liminf_{N \rightarrow \infty} \frac{\left| \{ p \leq N : \left( P_1( p - 1), \dots , P_l(p - 1), [\xi_1(p)], \dots , [\xi_k(p)]  \right) \in E - E \} \right| }{\pi(N)} > 0.$$
$$\liminf_{N \rightarrow \infty} \frac{\left| \{ p \leq N:  \left( P_1( p + 1), \dots , P_l(p + 1), [\xi_1(p)], \dots , [\xi_k(p)]  \right) \in E - E \} \right| }{\pi(N)} > 0.$$ 
\end{Theorem}

The structure of the paper is as follows. In Section 2 we establish various differential inequalities for functions in Hardy fields which are needed for the proofs of uniform distribution results in Section 3. We also collect in Section 2 various auxiliary number theoretical results. Section 4 is devoted to various applications. These include (some refinements of) the results formulated in this introduction as well as new results pertaining to sets of recurrence in $\mathbb{Z}^d$.

\subsection*{Notation} The following notation will be used throughout this paper.
\begin{enumerate}
\item We write $e(x) = \exp (2 \pi i x)$.
\item For positive $Y$, $X \ll Y$ (or $X= O(Y)$) means $ |X| \leq c Y $ for some positive constant $c$.
\item  $X \asymp Y$ means $c_1 X \leq Y \leq c_2 X$ for some positive constants $c_1, c_2$.
\item $\sum_{p\leq N}$ denotes the sum over primes in the interval $[1,N]= \{1,2, \dots, N\}$.
\item $\phi(n)$ is Euler's totient function, which is defined as the number of positive integers $\leq n$ that are relatively prime to $n$.   
\item  By $\pi(x)$, we denote the number of primes not exceeding $x$.
\item The M\"obius function $\mu: \mathbb{N} \rightarrow \mathbb{R}$ is defined by
$$\mu(n) = \begin{cases}
  0 & \text{if $n$ has one or more repeated prime factors} \\
  1 & \text{if $n=1$} \\
  (-1)^k  &\text{if $n$ is a product of $k$ distinct primes}. 
\end{cases}$$
\item The von Mangoldt function $\Lambda: \mathbb{N} \rightarrow \mathbb{R}$ is defined by
$$\Lambda(n) = \begin{cases}
  \log p & \text{if $n = p^k$ for some prime $p$ and integer $k \geq 1$} \\
  0 & \text{otherwise}. 
\end{cases}$$
\item $\Lambda_1(n)$ denotes the characteristic function of the set of primes,
 $$
\Lambda_1(n) = \left\{
        \begin{array}{ll}
            1 & \quad \textrm{if} \,\, n \, \text{is a prime number} \\
            0 & \quad \textrm{otherwise}.
        \end{array}
    \right.
$$
The function $\Lambda_1 (n, r, a)$ is defined by
$$
 \Lambda_1 ( n, r ,a ) =
\begin{cases}
1, & \text{if }n \,\, \text{is a prime number and} \, n \, \equiv a \, (\bmod \, r)  \\
0, & \text{otherwise}.
\end{cases}
$$
\item  For sets $A \subset B \subset \mathbb{Z}^m$, the relative density and the lower relative density of $A$ with respect to $B$ are defined as 
$$ \lim_{n \rightarrow \infty} \frac{|A \cap [-n,n]^m|}{|B \cap [-n,n]^m|} \quad \text{and} \quad  \liminf_{n \rightarrow \infty} \frac{|A \cap [-n,n]^m|}{|B \cap [-n,n]^m|}.$$ 
\item We write $f(x) \uparrow \infty$ if $f(x)$ is eventually increasing and $\lim\limits_{x \rightarrow \infty} f(x) = \infty$ and $f(x) \downarrow 0$ if $f(x)$ is eventually decreasing and $\lim\limits_{x \rightarrow \infty} f(x) = 0$.
\item By a slight abuse of notation we write $Tf(x) = f(Tx)$ for a measure preserving transformation $T$ and a function $f$.
\item $\{ x \}$ denotes the fractional part of a real number $x$.
\item The notation $\| \cdot \|$ must be handled with care: in Sections 2 and 3, for $x \in \mathbb{R}$, $\| x \|$ denotes the distance to the nearest integer of $x$ and in Section 4, $\|f\|_{\mathcal{H}} = \langle f, f\rangle^{1/2}$ is the norm of $f$ in a Hilbert space $\mathcal{H}$.
\end{enumerate}

\section{Preliminaries}
\label{sec2}
\subsection{Differential inequalities for functions from Hardy fields}\mbox{}

%We say that $f \in {\bf U}$ is {\it subpolynomial} if, for some $n \in \mathbb{N}$, $|f(x)| <  x^n$ eventually. 
Following Boshernitzan \cite{Bos} we say that a subpolynomial function $f(x) \in {\bf U}$ is of {\it type $x^{l+}$} $(l = 0, 1, 2, \dots )$ if 
$$\lim_{x \rightarrow \infty} \frac{x^l}{f(x)} = \lim_{x \rightarrow \infty} \frac{f(x)}{x^{l+1}}=0.$$

\begin{Proposition}[cf. Ch. VI in \cite{Har2}] 
For any function $f$ belonging to a Hardy field such that no derivative $f^{(n)}$ satisfies $|f^{(n)}| \asymp 1$,
\begin{equation}
\label{e0}
\left| \frac{x f^{(n+1)}(x)}{f^{(n)}(x)} \right| \gg \frac{1}{\log^2 x}.
\end{equation}
\end{Proposition}

\begin{proof}
Note that $\lim\limits_{x \rightarrow \infty} |f^{(n)} (x)|$ is either $0$ or $\infty$. Let $g(x) = f^{(n)} (x)$. Suppose \eqref{e0} does not hold, that is, for some $C >0$, 
$$\left| \frac{x g'(x)}{g(x)} \right| \leq \frac{C}{ \log^2 x}.$$

\noindent If $\lim\limits_{x \rightarrow \infty} |g(x)| = \infty$, then $\frac{g'(x)}{g(x)} \leq \frac{C}{x \log^2 x}$. Integrating on both sides from $a$ to $\infty$, where $a$ is some positive constant, we obtain
$$ \infty =  \lim\limits_{x \rightarrow \infty} \log |g(x)| -  \log |g(a)| = \int_a^{\infty} \frac{C}{x \log^2 x} < \infty,$$
 which gives a contradiction.

\noindent If $\lim\limits_{x \rightarrow \infty} |g(x)| = 0$, then $- \frac{g'(x)}{g(x)} \leq \frac{C}{x \log^2 x}$. Again, integrating on both sides leads to a contradiction: 
$$ \infty = \lim\limits_{x \rightarrow \infty}  (- \log |g(x)|) + \log |g(a)| = \int_a^{\infty} \frac{C}{x \log^2 x} < \infty.$$
\end{proof}

%The result in \cite{Har2} only considered logarithmico-exponential functions, but it only used some differentiation property to obtain those results, so the above proposition still works for any tempered function from a (general) Hardy field.

%proposition-Fejer
\begin{Proposition}
\label{Fejer}
Let $f(x) \in {\bf U}$ be a subpolynomial function. Suppose that 
\begin{equation}
\label{FC1}
f(x) \,\, \text{ is of the type} \,\, x^{0+} \quad \text{and} \quad \lim_{x \rightarrow \infty} \left| \frac{f(x)}{\log x}\right| = \lim_{x \rightarrow \infty} x |f'(x)| = \infty.
\end{equation}
 Then, for any $j \geq 1$ and sufficiently large $x$, we have
\begin{equation}
\label{e2}
\frac{1}{2 \log x} \ll  \frac{x f'(x)}{f(x)}  \ll 1,
\end{equation}
\begin{equation}
\label{e3}
\frac{1}{\log^2 x} \ll  \left| \frac{xf^{(j+1)} (x)}{f^{(j)}(x)} +j \right| \ll 1,
\end{equation}
\begin{equation}
\label{e4}
\frac{1}{\log^2 x} \ll  \left| \frac{xf^{(j+1)} (x)}{f^{(j)}(x)} \right|  \ll 1,
\end{equation}
\begin{equation}
\label{e1}
0 \ll  \frac{f'(x)}{f'(2x)}  \ll \log x.
\end{equation}
\end{Proposition}

%proof-Proposition1
\begin{proof}
Without loss of generality, we can assume that 
$$f(x) \uparrow \infty, \quad f'(x) \downarrow 0, \quad xf'(x) \uparrow \infty.$$
We start by proving \eqref{e2}. 
Consider $g(x) = f(x) - x f'(x)$. 
Then $g'(x) = - xf''(x)$. Since $f'(x) \downarrow 0$, $f''(x) < 0$, so $g'(x) > 0$ eventually. 
Hence there exists $K$ such that $g(x) > K$. Then $\frac{x f'(x)}{f(x)} \leq 1 - \frac{K}{f(x)}$. 
Since $f(x) \uparrow \infty$, we have $$\frac{x f'(x)}{f(x)} \ll 1.$$
Now consider $h(x) := x f'(x) \log x - f(x)$. 
Then $h'(x) = (x f''(x) + f'(x)) \log x$. 
Since $xf'(x)$ is increasing, $xf''(x) + f'(x) = (xf'(x))' \geq 0$. 
Thus, $h'(x) \geq 0$, $h(x) = x f'(x) \log x - f(x) > C$ for some constant $C$. 
Then $$\frac{x f'(x)}{f(x)} \geq \left(1+ \frac{C}{f(x)} \right) \frac{1}{\log x}.$$ 
Since $f(x) \uparrow \infty$, \eqref{e2} follows.

\noindent We turn now our attention to formulas \eqref{e3} and \eqref{e4}.
By L'Hospital's rule,
$$\lim_{x \rightarrow \infty} \frac{x f'(x)}{f(x)} = \lim_{x \rightarrow \infty} \frac{x f^{(j+1)}(x)}{f^{(j)}(x)} + j.$$
From \eqref{e2}, 
$$ \left| \frac{xf^{(j+1)} (x)}{f^{(j)}(x)} +j \right| \ll 1, \quad  \left| \frac{xf^{(j+1)} (x)}{f^{(j)}(x)} \right|  \ll 1. $$
From \eqref{e0}, 
$$ \frac{1}{\log^2 x} \ll  \left| \frac{xf^{(j+1)} (x)}{f^{(j)}(x)} \right|.$$
Now let $q(x) = x^j f^{(j)}(x)$. 
Note that from $\lim\limits_{x \rightarrow \infty} x f'(x) = \infty$, by L'Hospital's rule
$$ \lim_{x \rightarrow \infty} |x^2 f''(x)| = \lim_{x \rightarrow \infty} \left|\frac{f'(x)}{1/x}\right| = \infty.$$
By a similar argument, we have $\lim\limits_{x \rightarrow \infty} |x^k f^{(k)}(x)| = \infty$ for all $k$. 
Now, $\lim\limits_{x \rightarrow \infty} f'(x) =0$ implies $\lim\limits_{x \rightarrow \infty} |x^m f^{(n)}(x)| =0$ for any non-negative integers $m < n$. 
Thus no $q^{(n)}$ satisfy $q^{(n)} \asymp 1$.  
Then from \eqref{e0},
$$\frac{1}{\log^2 x} \ll \left| \frac{x g'(x)}{g(x)}\right| = \left| \frac{x f^{(j+1)}(x)}{f^{(j)}(x)} + j \right|.$$

\noindent Finally, let us prove \eqref{e1}. 
From \eqref{e2}, we have
$$\frac{f(x)}{2 x \log x} \ll  f'(x)  \ll \frac{f(x)}{x}. $$
Hence 
$$ \frac{f'(x)}{f'(2x)} \ll \frac{f(x)}{x} \cdot \frac{2 (2x) \log 2x}{f(2x)} =4 \frac{f(x)}{f(2x)} \log 2x \leq 4 \log 2x, $$
since $f(x) \leq f(2x)$.
\end{proof}

\begin{Remark} 
The proof actually shows that in the formulas \eqref{e0}, \eqref{e3} and \eqref{e4} one can replace $\log^2 x$ with $\log^{1 + \epsilon} x$ for any $\epsilon > 0$. 
\end{Remark}

\begin{Proposition}
\label{tempered}
Let $f(x) \in {\bf U}$ be a subpolynomial function. 
If $f(x)$ is of the type $x^{l+}$ for some $l \geq 1$, then for any non-negative integer $j$ and any $\epsilon >0$
\begin{equation}
\label{e5}
\left| j + \frac{x f^{(j+1)} (x)}{f^{(j)} (x)}\right| \asymp 1
\end{equation}
and
\begin{equation}
\label{e6}
x^{\beta - j - \epsilon} \ll |f^{(j)}(x)| \ll x^{\beta-j+\epsilon}
\end{equation}
for some fixed $\beta \in [l,l+1]$.
\end{Proposition}

\begin{proof}
Without loss of generality we can assume that $f(x) >0$ eventually.
First, by L'Hospital's rule,
\begin{equation}
\label{e7}
\beta = \lim_{x \rightarrow \infty} \frac{\log f(x)}{\log x} = \lim_{x \rightarrow \infty} \frac{xf'(x)}{f(x)}.
\end{equation} 
Since 
$$\lim_{x \rightarrow \infty} \frac{x^l}{f(x)} = \lim_{x \rightarrow \infty} \frac{l!}{f^{(l)} (x)} = 0 \quad \text{and} \quad
\lim_{x \rightarrow \infty} \frac{f(x)}{x^{l+1}} = \lim_{x \rightarrow \infty} \frac{f^{(l+1)} (x)}{(l+1)!} = 0,$$
we have $x^{l} \ll f(x) \ll x^{l+1}$, which proves $l \leq \beta \leq l+1$. 

\noindent Again, by L'Hospital's rule, if $\lim\limits_{x \rightarrow \infty} \frac{(xf'(x))^{(j)}}{(f(x))^{(j)}} = \beta$, then $\lim\limits_{x \rightarrow \infty} \frac{(xf'(x))^{(j+1)}}{(f(x))^{(j+1)}} = \beta$
since
\begin{enumerate}
\item $\lim\limits_{x \rightarrow \infty} f^{(j)} (x) = \pm \infty$ or $0$ from that $f(x)$ is of the type $x^{l+}$.
\item $\lim\limits_{x \rightarrow \infty} f^{(j)} (x) =\lim\limits_{x \rightarrow \infty} (xf'(x))^{(j)}= \pm \infty$ or $0$ from that $0 < \beta < \infty$.
\end{enumerate}
Note now that $$\frac{(xf'(x))^{(j)}}{f^{(j)}(x)} = j + \frac{x f^{(j+1)}(x)}{f^{(j)}(x)},$$
which proves \eqref{e5}. 
Also, \eqref{e7} implies that for any $\epsilon >0$ we have 
$$x^{\beta- \epsilon} \ll f(x) \ll x^{\beta + \epsilon}.$$

\noindent Now we claim that for any $a \ne 0$,
\begin{equation}
\label{eq:a,b}
\lim_{x \rightarrow \infty} \frac{f^{(j)}(x)}{x^{a}} = \lim_{x \rightarrow \infty} \frac{f^{(j+1)}(x)}{a x^{a-1}}.
\end{equation}
Choose $b \in \mathbb{R}$ such that both $f^{(j)}(x) x^b$ and $x^{a+b}$ converge to $0$ as $x \rightarrow \infty$. 
Then by applying L'Hosptial's rule we get
$$\lim_{x \rightarrow \infty} \frac{f^{(j)}(x)}{x^a} = \lim_{x \rightarrow \infty} \frac{x^b f^{(j)}(x)}{x^{a+b}} = \lim_{x \rightarrow \infty} \frac{b}{a+b} \frac{f^{(j)}(x)}{x^a} + \lim_{x \rightarrow \infty} \frac{1}{a+b} \frac{f^{(j+1)}(x)}{x^{a-1}}.$$
Thus,
$$\frac{a}{a+b}\lim_{x \rightarrow \infty} \frac{f^{(j)}(x)}{x^a} =  \frac{1}{a+b} \frac{f^{(j+1)}(x)}{x^{a-1}}.$$
Dividing both sides with $\frac{a}{a+b}$, we obtain \eqref{eq:a,b}.
Hence, $$x^{\beta-j - \epsilon} \ll |f^{(j)} (x)| \ll x^{\beta-j+ \epsilon},$$ which proves \eqref{e6}.
\end{proof}

%section-preliminary lemmas
\subsection{Some classical number-theoretic lemmas}\mbox{}

We collect in this subsection some classical results which will be needed for the proofs in the next section. 
Recall that $\pi(x)$ is the number of primes less than or equal to $x$.
\begin{Lemma}[Erd\H os-Tur\'an; cf. Theorem 1.21 in \cite{DT} or Theorem 2.5 in \cite{KN}]\mbox{}
\label{lem1}
For a sequence $(a_n)_{n \in \mathbb{M}}$ of real numbers, we define 
 $$D(x) := \sup_{[\alpha, \beta] \subset [0,1]} \left| \frac{1}{\pi(x)} \big| \{ p \in [1,x] : \{a_p\} \in [\alpha, \beta] \}\big| - (\beta- \alpha)\right|.$$
For any $Q$ we have
$$D(x) \ll \frac{1}{Q} + \frac{1}{\pi(x)} \sum_{q \leq Q} \frac{1}{q} \left| \sum_{p \leq x} e(q a_p) \right|.$$
\end{Lemma}

\begin{Lemma}
\label{lemmaS}
Let $(a_n)_{n \in \mathbb{N}}$ be a sequence of real numbers. 
Suppose that for any large enough $X$, there exists $Q=Q(X) \uparrow \infty$ such that for all $1 \leq q \leq Q$,
\begin{equation}
\label{eqS} 
S:= \sum_{X_0 < p \leq X} e(q a_p)  \ll \frac{\pi(X)}{Q},
\end{equation}
where $X_0 = \frac{X}{Q}$.
Then $(a_p)_{p \in \mathcal{P}}$ is u.d. $\bmod \, 1$.
\end{Lemma}
\begin{proof}
Let
$$D(X) := \sup_{[\alpha, \beta] \subset [0,1]} \left| \frac{1}{\pi(x)} |\{ p \in [1,x] : \{a_p\} \in [\alpha, \beta] \}| - (\beta- \alpha)\right|.$$
Let $Q_1 = Q_1(X) = \frac{Q}{\log Q}$. Then $\lim\limits_{X \rightarrow \infty} Q_1 (X) = \infty$.
From \eqref{eqS},
\begin{equation}
\label{**}
\frac{1}{\pi(X)} \sum_{q \leq Q} \frac{1}{q}  \left| \sum_{X_0 < p \leq X} e(q a_p) \right|  \ll \frac{1}{Q_1}.
\end{equation}
Since
$$ \frac{1}{\pi(X)} \sum_{q \leq Q} \frac{1}{q} \left| \sum_{p \leq X_0} e(q a_p) \right| \ll \frac{1}{\pi(X)} \sum_{q \leq Q} \frac{1}{q} \, \pi \left(\frac{X}{Q}\right) \ll \frac{\log Q}{Q},$$
we have
\begin{equation}
\label{*}
\frac{1}{\pi(X)} \sum_{q \leq Q} \frac{1}{q}  \left| \sum_{p \leq X} e(q a_p) \right| \ll \frac{1}{Q_1}. 
\end{equation} 
By Erd\H os-Tur\'an inequality (Lemma \ref{lem1}), 
$$D(X) \ll \frac{1}{Q} + \frac{1}{Q_1} \rightarrow 0 \quad \text{as} \quad X \rightarrow \infty. \qedhere$$
\end{proof}

\begin{Lemma}[Prime number theorem] \mbox{}
\label{lem2}
\begin{enumerate}
\item $($cf.\ \cite[Theorem 2 in Ch.\ V]{Kar}$)$  
$$\pi(x) := \sum_{p \leq x} 1 = Li(x) + R(x),$$
where $Li(x) = \int_2^x \frac{dt}{\log t}$ and $R(x) \ll x \exp(-C \sqrt{\log x})$ for some positive constant $C$.
\item $($cf.\ \cite[Theorem 6 in Ch.\ IX]{Kar}$)$
Let $(a,q) =1$. 
By $\pi(x;q,a)$ we denote the number of primes less than or equal to $x$ that are congruent to $a \, \bmod \, q$.
Then, for some positive constant $C$,
$$\pi(x;q,a) = \frac{1}{\phi(q)} Li(x) - E \frac{\chi(a)}{\phi(q)}\int_2^x \frac{u^{\beta_1 -1}}{\log u} \, du + O (x \exp (-C \sqrt{\log x})),$$
where $E=1$ if there exists a real character $\chi$ modulo $q$  such that $L( s, \chi)$ has a real zero $\beta_1$ in $(1- c/ \log q ,1 )$ (so-called Siegel zero) and $E=0$ otherwise.
\end{enumerate}
\end{Lemma}

The following lemma is a corollary of the previous result.
\begin{Lemma}[Siegel - Walfisz, cf.\ Corollary 2 in Ch. IX \cite{Kar}] \mbox{}
\label{SW}
For every $A >0$, there exists $C_A >0$ such that for any $x$, $2 \leq q \leq (\log x)^A$ and $(a,q)=1$,
$$\pi(x;q,a) = \frac{1}{\phi(q)} Li(x) + O (x \exp (-C_A \sqrt{\log x})).$$  
\end{Lemma}

\begin{Lemma}[Partial summation formula]\mbox{}
\label{lem5}
Let $(a_n)_{n \in \mathbb{N}}$ and $(b_n)_{n \in \mathbb{N}}$ be two sequences.
For any non-negative integers $X_1, X_2$ with $X_1 < X_2$
$$\sum_{n=X_1}^{X_2} a_n b_n = \sum_{n = X_1}^{X_2 - 1} \left(  (a_n - a_{n+1}) \sum_{m=X_1}^n b_m \right) + a_{X_2} \sum_{m=X_1}^{X_2} b_m.$$
Also, let $\sum\limits_{m =X_1}^{x}  b_m = B(x) + R(x)$ with $|R(x)| \leq R$ for  $X_1-1 \leq x \leq  X_2$  and let $|a_n | \leq A$. 
Then 
$$ \sum_{n= X_1}^{X_2}   a_n b_n  =  \sum_{n=X_1}^{X_2}  a_n ( B(n) - B ( n-1)) + O ( RA + R \sum_{n= X_1}^{X_2-  1} |a_n-a_{n+1}|).$$
\end{Lemma}

\begin{proof} 
The first part is a well-known partial summation formula. 
To prove the second part, we write
\begin{align*}
 \sum_{n= X_1}^{X_2} a_n b_n &= \sum_{n=X_1}^{X_2} a_n ( B(n) + R(n) - B(n-1) - R(n-1))  \\
                                                     &= \sum_{n = X_1}^{X_2} a_n( B(n) - B(n-1) ) + \sum_{n=X_1}^{X_2} a_n ( R(n)- R(n-1) ).
\end{align*}
Then the second sum can be estimated as follows:
\begin{align*}
\left| \sum_{n=X_1}^{X_2} a_n ( R(n)- R(n-1) ) \right| &= \left| \sum_{n=X_1}^{X_2 -1}  ( a_n - a_{n+1} ) R(n) +  a_{ X_2} R(X_2)  - a_{X_1} R( X_1 - 1) \right| \\
&\leq 2AR + R \sum_{n= X_1}^{X_2-1}  |a_n - a_{n+1}|.
\end{align*}
This completes the proof of the second part of the lemma.
\end{proof}

\begin{Lemma}[Weyl - van der Corput; see Lemma 2.5 in \cite{GK}] \mbox{}
\label{lem6}
Suppose that $(\xi(n))_{n \in \mathbb{N}}$ is a complex valued sequence.
For any positive integer $H$ and any interval $I$ in $\mathbb{N}$ of the form $(a,b]=\{a+1, a+2, \dots, b\}$, we have 
$$\left| \sum_{n \in I} \xi(n) \right|^2 \leq \frac{|I| + H}{H} \sum_{|h| \leq H} \left( 1- \frac{|h|}{H} \right) \sum_{n,n+h \in I} \xi(n) \overline{\xi(n+h)}.$$
\end{Lemma}

\begin{Lemma}[\cite{Va}]\mbox{}
\label{lem7}
For any positive integers $u, v$ and any $X$ we have
$$\sum_{v \leq n \leq X} \Lambda(n) g(n) = T_1 - T_2 - T_3,$$
where
\begin{align*}
T_1 &= \sum_{d \leq u} \sum_{m \leq X/d} \mu(d) (\log m) g(dm),\\
T_2 &= \sum_{m \leq uv} \sum_{r \leq X/m} a(m) g(mr), \quad a(m) = \sum\limits_{d \leq u} \sum\limits_{\substack{n \leq v \\ dn=m}} \mu(d) \Lambda(n)\\
T_3 &= \sum_{m > u} \sum_{v < n \leq X/m} b(m) \Lambda(n) g(mn), \quad b(m) = \sum\limits_{\substack{d \leq u \\ d | m}} \mu(d). 
\end{align*}
\end{Lemma}

\begin{Remark}
For $a(m)$ and $b(m)$ as defined in Lemma \ref{lem7}, the following estimates will be useful:
\label{etc}
\begin{align*}
&\bullet \sum_{m \leq uv} \frac{|a(m)|}{m} =  \sum_{m \leq uv} \frac{1}{m} \sum\limits_{d \leq u} \sum\limits_{\substack{n \leq v \\ dn=m}} \mu(d) \Lambda(n)
                                                         \ll  \sum_{d \leq u} \frac{|\mu(d)|}{d} \sum_{n \leq v} \frac{\Lambda(n)}{n} \ll \log u \, \cdot \, \log^2 v. \\
&\bullet \sum_{y \leq m \leq 2y} |b(m)|^2 = \sum_{y \leq m \leq 2y} \big| \sum\limits_{\substack{d \leq u \\ d | m}} \mu(d) \big|^2 = \sum_{d_1, d_2 \leq u} \mu(d_1) \mu(d_2) \sum_{\substack{y \leq m \leq 2y \\ d_1|m, d_2|m}} 1 \\
       &\quad \quad \quad \quad \quad \quad \quad = \sum_{d \leq u}  \sum_{d_1 \leq \frac{u}{d}} \sum_{d_2 \leq \frac{u}{d}} \mu(d d_1 d_2) \frac{2y-y}{d d_1 d_2} 
       \leq y \sum_{d, d_1, d_2 \leq u} \frac{1}{d d_1 d_2} \ll y \log^3 u.\\
&\bullet \sum_{y \leq m, m+h \leq 2y} |b(m)| \, |b(m+h)| \ll \sum_{ m } (b^2(m) + b^2(m+h)) \ll y \log^3 u.\\
&\bullet \sum_{y \leq n \leq 2y} \Lambda^2(n) \ll \sum_{y \leq p \leq 2y} \log^2 p + \sqrt{y} \, \log^2 y \ll y \log y. \\
&\bullet \sum_{y \leq n, n+h \leq 2y} \Lambda(n) \Lambda (n+h) \leq \frac{1}{2} \sum_{ n}( \Lambda^2(n) + \Lambda^2 (n+h) ) \ll y \log y. \quad \quad \quad \quad \quad
\end{align*}
\noindent For the estimate of $\sum \Lambda^2 (n)$, we used the following observation 
 $$| \{p^k: y \leq p^k \leq 2y, k \geq 2 \}| \leq |\{n \in \mathbb{N}: n^k \leq 2y, k \geq 2 \}| \leq \sqrt{y} \, \log y$$  and the fact that $\Lambda(p^k) \leq \log 2y$ if $p^k \leq 2y$. 
\end{Remark}

\subsection{Some auxiliary results regarding exponential sums}\mbox{}

\begin{Lemma} [cf. p.34 in \cite{Mo}]
\label{Mo}
 \begin{equation}
 \label{mo}
 \frac{1}{q} \sum_{j=1}^q e\left(\frac{(n-b)j}{q} \right)= 
 \begin{cases}
 1 &  n \equiv b \,\, (\bmod \, q)\\
 0 &  \text{otherwise}
 \end{cases}
\end{equation} 
for any $q \in \mathbb{N}$ and $b \in \mathbb{N}$ with $1 \leq b \leq q$.
\end{Lemma}

In the following lemma, $I$ denotes an interval in $\mathbb{N}$ of the form $(a, b] = \{a+1, a+2, \dots , b\}$, where $(a, b \in \mathbb{N})$. 
Note that in the formulation of Lemma \ref{lem3} $\| \cdot \|$ denotes the distance to the closest integer.
\begin{Lemma}[Kusmin - Landau; see Theorem 2.1 in \cite{GK}]
\label{lem3}
If $f(x)$ is continuously differentiable, $f'(x)$ is monotonic, and $0 < \lambda \leq \| f'(x) \| \leq \frac{1}{2}$ on some interval $I$, 
then
$$ \left| \sum_{n \in I}  e(f(n)) \right| \ll \frac{1}{\lambda}.$$
\end{Lemma}

%\begin{Lemma}[see Theorem 2.2 and Theorem 2.8 in \cite{GK}]
%\label{lem4}
%Let $k$ be a non-negative integer. Suppose that $f(x)$ is a real valued function with $(k+2)$-times continuous derivatives on some interval $I$ such that for some $\alpha \geq 1$ and %$\lambda >0$,
%$$\lambda \leq |f^{(k+2)}(x)| \leq \alpha \lambda$$
%on $I$. Denote $K = 2^k$. Then
%$$\left| \sum_{n \in I} e(f(n)) \right| \ll |I| (\alpha^2 \lambda)^{\frac{1}{4K-2}} + |I|^{1- \frac{1}{2K}} \, \alpha^{\frac{1}{2K}} + |I|^{1- \frac{2}{K} + \frac{1}{K^2}} \, \lambda^{- %\frac{1}{2K}}.$$
%In particular, when $k=0$ we have
%$$ \left| \sum_{n \in I}  e(f(n)) \right| \ll \alpha |I| \lambda^{1/2} + \lambda^{-1/2}.$$
%\end{Lemma}

The following lemma is an improvement of Lemma 2.5 in \cite{BKMST}.
\begin{Lemma}
\label{lem2.3.1}
 Let $k \geq 1$ and denote $K = 2^k$. Let $f(x)$ be a $(k+1)$-times continuously differentiable real function on $I = ( X_1, X_1+ X] \subset (X_1, 2 X_1]$, where $X, X_1 \in \mathbb{N}$. 
Assume that $f^{(k+1)}(x)$ is monotone on $I$ and $\lambda \leq |f^{(k+1)}(x)| \leq \alpha \lambda$ for some $\lambda, \alpha>0$. 
Then we have 
\begin{equation}
\label{nov18}
\begin{split}
S &:= \left| \sum_{x \in I} e(f(x))\right| \\
   & \ll X \left[(\alpha \lambda)^{1/(2K-2)} + (\lambda X^{k+1})^{-1/K} (\log X)^{k/K} + \left(\frac{\alpha \log^k X}{X}\right)^{1/K} \right],
\end{split}
\end{equation}
where the implied constant depends on $k$ only. 
%Also, 
%\begin{equation}
%\label{(1)}
 %S \ll  X \left[ ( \alpha \lambda ) ^{ 1/ K k } + ( \lambda  X^{k+1} )^{-1/ 2 K} + ( \alpha / X )^{ 1/K} \log X \right].
%\end{equation}
\end{Lemma}
To prove Lemma \ref{lem2.3.1} we need the following result, which follows from Lemma \ref{lem6} by iteration. 
\begin{Lemma}[{cf. \cite[Lemma 2.7]{GK}}]
\label{lemma2.10}
Let $k$ be a positive integer and $K = 2^k$. 
Assume that $I = (X_1, X_1 + X] \subset (X_1, 2X_1]$ and let $S = \left| \sum\limits_{ x \in I } e(f(x)) \right| $.
For any positive $H_1, \dots , H_k \ll X$, we have
\begin{multline}\label{(2)}
\left( \frac{S}{X}\right)^K \leq 8^{K-1} \left\{  \frac{1}{H_1^{K/2}} + \frac{1}{H_2^{K/4}} + \cdots + \frac{1}{H_k} + \right.\\
\left. \frac{1}{H_1 \cdots H_k   X} \sum_{h_1 = 1}^{H_1} \dots \sum_{h_k = 1}^{H_k}  \left|  \sum_{x \in I(\underline{h})} e(f_1(x)) \right| \right\},
\end{multline}
where 
$f_1(x) := f(\underline{h}, x) = h_1 \cdots h_k \int_0^1 \cdots \int_0^1 \frac{\partial^k}{\partial x^k} f(x+ \underline{h} \cdot \underline{t} ) \, d \underline{t}$, 
$\underline{h} = (h_1, \cdots , h_k)$, $\underline{t} = (t_1, \cdots , t_k)$ and $I(\underline{h}) = (X_1, X_1 + X - h_1 - \cdots - h_k]$.  
\end{Lemma} 

\begin{proof}[Proof of Lemma \ref{lem2.3.1}]
%We start with proving \eqref{nov18}.
Using first \eqref{(2)} and then Lemma \ref{lem3} we obtain 
\begin{align}
\label{(3)}
| S / X |^K &\ll ( 1/H_1)^{(K/2)} + \cdots + 1/H_k + 1/ ( H_1 \cdots H_k X )  \sum_{h_1=1}^{H_1} \cdots \sum_{h_k=1}^{H_k}   \frac{1}{\lambda h_1 \cdots h_k}  \notag \\
                  &\ll (1/H_1)^{(K/2)} + \cdots + 1/ H_k + ( \log X )^k /( \lambda  X  H_1 \cdots H_k).
\end{align}
Now we need to minimize the last expression subject to the conditions
\begin{equation}
\label{eq:sub} 
H_ j \leq X/(2 k) \,  (j=1, \dots, k) \quad \text{ and } \quad 2 \alpha  \lambda  H_1 \cdots H_k  \leq1.
\end{equation}

\noindent If $\alpha \lambda X^k \leq 1$, then we take  $H_j = X/ (2k)$ and get  
$$|S/X|^K \ll 1/X + (\log X)^k  / (\lambda  X^{(k+1)}),$$
which implies \eqref{nov18}.
Assume now that $\alpha  \lambda  X^k >1$. 
If $2 \alpha  \lambda X^{(2K - 2)/K} \geq 1$, that is, if $1/X \leq (2 \alpha \lambda)^{K/(2K-2)}$, 
then we take $ H_1 = (1/( 2 \alpha \lambda ))^{1/(K-1)} \leq X^{2/K}$ and $H_ j   = ( H_1)^{2^{j-1}}  (j=2, \dots , k )$.
Note that $ 2 \alpha \lambda H_1 H_2 \cdots H_k = 2 \alpha \lambda  H_1 ^ { 1 + 2 + 2^2 + \cdots+ K/2} = 1$, 
so that \eqref{eq:sub} is satisfied and  the last term in \eqref{(3)} is $2 \alpha \frac{\log^k X}{X}$.
Hence, we obtain
$$  |S/X |^K \ll ( \alpha  \lambda )^{  K/ (2 K-2 )} + \alpha  \log^k X /X,$$
which proves \eqref{nov18}  in case $\alpha \lambda X^k >1$ and $2 \alpha  \lambda X^{(2K - 2)/K} \geq 1$.

\noindent To complete the proof of \eqref{nov18}, we need to establish the desired estimate for the case $1 /  X^k \leq 2 \alpha \lambda \leq  X^{(2- 2K)/K}$. 
Assume that, for some $j = 1, 2, \dots ,k-1$,
\begin{equation}
\label{(4)}
 X^{ \frac{2J}{K}-j-2} \leq 2 \alpha \lambda \leq X^{\frac{J}{K} -j-1},
\end{equation}
where $J = 2^j$. 
This covers all the remaining possibilities. 
Take $H_k = \cdots = H_{k-j+1} = X/(2 k ), H_2 =( H_1)^2 , H_3 = (H_1)^4,   H_{k-j}=(H_1)^{K/(2J)}$ so that
$$2 \alpha \lambda X^j H_1 H_2 \cdots H_{k-j} = 2 \alpha \lambda X^j   H_1^{ 1+2+... + K/(2J)} =1,$$ 
so that
$ H_1= (2 \alpha \lambda X^j )^{  J /( J-K)} \leq X^{2J/K}$. 
Then $2 \alpha \lambda H_1 \cdots H_k = (1 /2k)^j$, 
and the last term in \eqref{(3)} is $2 \alpha \frac{1}{(2k)^j}\frac{\log^k X}{X}$.
Thus \eqref{(3)} will acquire the following form:
$$  |S/X|^K \ll  (\alpha \lambda X^j )^{KJ/(2K-2J)} + \alpha \log^k X /X.$$
We will show now with the help of \eqref{(4)} that $(\alpha \lambda X^j )^{J / ( 2K-2J )} \ll (\alpha \lambda )^{1/ ( 2K-2 )}$.
Indeed we have 
\begin{align*}
(\alpha \lambda X^j )^{J / ( 2K-2J )} \ll (\alpha \lambda )^{1/ ( 2K-2 )} &\Leftrightarrow (\alpha \lambda X^j)^{J(2K-2)} \ll (\alpha \lambda)^{2K -2J} \\
&\Leftrightarrow (\alpha \lambda)^{2KJ + 2J} X^{jJ(2K-2)} \ll 1 \\
&\Leftrightarrow \alpha \lambda X^{j(K-1)/(K+1)} \ll1.
\end{align*}
Since 
$$ \alpha \lambda X^{j(K-1)/(K+1)} \leq \alpha \lambda X^j \leq X^{J/K -1} \leq 1,$$
the proof of the lemma is completed.
\end{proof}

\begin{Lemma} [Theorem 1, p.133 \cite{Vin}]
\label{lem2.3.3}
Let $f(x)= \alpha_1 x + \cdots + \alpha_t x^t$, where one of the $\alpha_i \, (1 \leq i \leq t)$ is irrational.
For a given positive integer $X$, write all coefficients of $f(x)$ in the form 
\begin{equation}
\begin{split}
\alpha_j = \frac{e_j}{q_j} + z_j \quad \text{with} \quad (e_j, q_j) =1, \quad 0 < q_j \leq \tau_j := X^{j/2} \\
|z_j| \leq \frac{1}{q_j \tau_j}, \quad z_j = \delta_j X^{-j} \quad (j=1,2, \dots, t).
\end{split}
\end{equation}
Let 
\begin{equation}
%\label{deltaQ}
\delta_0 = \max \{ |\delta_1|, \dots , |\delta_t| \} \quad \text{ and} \quad \mathcal{Q}=[ q_1,\dots , q_t ].
\end{equation}
Denote $$\nu=\frac{1}{t}, \quad \rho= \frac{1}{ 17 t^2 ( 2 \log t + \log \log t +2.8)}, \quad t \geq 12,$$ 
$$H= \exp \left(\frac{(\log \log X)^2 }{ \log (1 + \epsilon_0)}\right)$$ 
for some small positive $\epsilon_0$. 

If $\mathcal{Q} \leq  X^{\nu/5}$ and $\delta_0 \leq X^{\nu}$, 
put $\Delta_1 = (m,\mathcal{Q})^{\nu/2} \mathcal{Q}^{- \nu/2 + \epsilon_0}$ if $\delta_0 < 1$ 
and  $\Delta_1 = \delta_0^{- \nu/2} \mathcal{Q}^{ - \nu/2 + \epsilon_0}$ if $\delta_0 \geq 1$; 
otherwise put $\Delta_1 = X^{- \rho}$. 

Assuming  $m \leq ( \Delta_1)^{-2}$, we have 
$$ \left| \sum_{ p \leq X}  e(mf(p)) \right| \ll HX \Delta_1.$$
\end{Lemma}

%section-equidistribution
\section{Uniform distribution}
\label{sec3}

%\begin{Theorem} [\cite{Bos}]
%\label{Bos}
%For a subpolynomial function $f \in {\bf U}$, the following two conditions are equivalent:
%\begin{enumerate}
%\item $(f(n))_{n \in \mathbb{N}}$ is u.d. $\bmod \, 1$.
%\item $\lim\limits_{x \rightarrow \infty} \frac{f(x) - P(x)}{\log x} = \pm \infty, \quad \forall P(x) \in \mathbb{Q}[x]$.
%\end{enumerate} 
%\end{Theorem}

In this section we prove the main theorem of this paper:
\begin{Theorem} 
\label{main} For a subpolynomial function $f \in {\bf U}$, the following conditions are equivalent:
\begin{enumerate}
\item $(f(n))_{n \in \mathbb{N}}$ is u.d. $\bmod \, 1$.
\item $(f(p))_{p \in \mathcal{P}}$ is u.d. $\bmod \, 1$.
\item For any $P(x) \in \mathbb{Q}[x]$,
\begin{equation}
\label{Bcondition} 
 \lim_{x \rightarrow \infty} \frac{f(x) - P(x)}{\log x} = \pm \infty.
\end{equation}
\end{enumerate}
\end{Theorem}

As was already mentioned in the introduction, the equivalence of conditions (1) and (3) was established in \cite{Bos}. So it is sufficient to show $(2) \Leftrightarrow (3)$. The proof of the implication $(2) \Rightarrow (3)$ is relatively routine and is given at the end of this section.
The main task is to show $(3) \Rightarrow (2)$.
In order to do this, we will use the classical method of reducing the evaluation of the sums $\sum\limits_{X_0 < p \leq X} e(qf(p))$, 
where $f(x) \in {\bf U}$ and $q \in \mathbb{Z} \backslash \{0\}$, 
to the estimation of expressions of the form $\sum \Lambda(n) e(qf(n))$. 
Indeed, it follows from Lemma \ref{lem5} that 
\begin{align*}
\sum_{X_0 < p \leq X} e(qf(p)) &= \sum_{X_0 < n \leq X} \frac{1}{\log n} \Lambda(n) \, e(qf(n)) + O (\sqrt{X}) \\ 
 &\ll \frac{1}{\log X} \max_{X_0 < X_1 \leq X} \left|\sum_{X_0 \leq n \leq X_1} \Lambda(n) \, e(qf(n)) \right| + \sqrt{X}.
\end{align*}
The task of estimating the sums $\sum \Lambda(n) e(qf(n))$ can, in turn, be reduced to estimating the classical expressions $\sum e(qf(n))$.
%{\color{red} The implementation of this reduction depends, however, on the rate of growth of $f(x)$.} 
In general, a subpolynomial function $f(x) \in \bf{U}$ can be written as $f(x) = f_1(x)+ f_2(x)$, where $f_1(x)$ is of type $x^{l+}$ for some $l \geq 0$ and $f_2(x) \in \mathbb{R}[x]$.
If $f_1(x)$ is of type $x^{l+}$ for some $l \geq 1$ or $f_1(x)$ is of type $x^{0+}$ with an appropriate growth rate 
(this, in our context, means roughly that for some $c>0$ and all large $x$, $|f_1'(x)| \geq \frac{\log^c x}{x}$), 
then Lemmas \ref{lem3} and \ref{lem2.3.1} provide a good estimate for $\sum_{n} e(q f(n))$.
Then Vaughan's identity (Lemma \ref{lem7}) and the estimates listed in Remark \ref{etc} allow one to get the desired estimate for $\sum \Lambda(n) e(qf(n))$.
On the other hand, if $f_1(x)$ is of ``slow growth" one needs to employ other methods 
- the prime number theorem or Vinogradov's lemma (Lemma \ref{lem2.3.3}) - to estimate $\sum\limits_{ p \leq X} e(q f(p))$. 

Accordingly, the proof of Theorem \ref{main} will be achieved by separately considering for our function $f= f_1 + f_2$ the following four cases:
%Thus we will divide the following 4 cases for $f(x) \in {\bf U}$ with condition \eqref{Bcondition} and we will apply appropriate methods to each case.  
\begin{enumerate}[(a)]
\item $f_1(x)$ is of the type $x^{0+}$ and $\lim\limits_{x \rightarrow \infty} \frac{f(x)}{\log x} = \pm \infty$ and $f_2(x) =0$ (Theorem \ref{thm1}).
\item $f_1 (x)$ is of the type $x^{l+}$ for some $l \geq 1$ and $f_2(x) =0$ (Theorem \ref{thm2}).
\item $f_1(x)$ is as in (a) and (b) and $f_2(x) \in \mathbb{R}[x]$ (Theorem \ref{thm3}).
\item $f_1(x)$ is of the type $x^{0+}$ and $\lim\limits_{x \rightarrow \infty} \left| \frac{f_1(x)}{\log x}\right| < \infty$ and $f_2(x) \in \mathbb{R}[x]$ such that one of coefficients of $f_2(x) - f_2(0)$ is irrational (Theorem \ref{thm4}).
\end{enumerate}

%Theorem1-Fejer functions
\begin{Theorem}
\label{thm1}
Let $f(x) \in {\bf U}$ be a subpolynomial function. 
Suppose that 
\begin{equation}
\label{FC}
f(x) \,\, \text{ is of the type} \,\,x^{0+} \quad \text{and} \quad \lim_{x \rightarrow \infty} \left| \frac{f(x)}{\log x}\right| = \lim_{x \rightarrow \infty} x |f'(x)| = \infty.
\end{equation}
Then $(f(p))_{p \in \mathcal{P}}$ is u.d. mod 1.
\end{Theorem}

%proof for Theorem 1
\begin{proof}
In light of Lemma \ref{lemmaS}, it is sufficient to show that  if $Q(X) \leq \log X$ is an unbounded positive increasing function with 
\begin{equation}
\label{eqn3.1} 
 Q \leq \frac{1}{2} \left| f' \left( \frac{X}{ \log X}\right) \right|^{-1} \,\, \textrm{and}  \,\, Q \leq |Xf'(X)|,
\end{equation}
then  
$$ S:= \sum_{X_0 < p \leq X} e(qf(p)) = \sum_{X_0 < n \leq X} \Lambda_1(n) e(qf(n)) \ll \frac{\pi(X)}{Q},$$
where $X_0 = \frac{X}{Q}$ and $\Lambda_1(n)$ is the characteristic function of the set of prime numbers.

Without loss of generality, we assume that $f(x) \geq 0$ eventually, so $f(x) \uparrow \infty, f'(x) \downarrow 0$ and $x f'(x) \uparrow \infty$.
We split the sum $S$ as follows 
$$S = S_0 + S_1 = \sum_{X_1 < p \leq X}  e(qf(p)) + \sum_{X_0 < p \leq X_1} e (qf(p)),$$
where $f'(n) \leq \frac{\log^{15} X}{X}$ for $n > X_1$ and $f'(n) >  \frac{\log^{15} X}{X}$ for $n \leq X_1$.

Using Lemma \ref{lem5} with $a_n= e(qf(n))$ and $b_n= \Lambda_1(n)$ we obtain (upon invoking Lemma \ref{lem2})
\begin{align}
\label{S_0}
S_0  :=&\sum_{X_1 < p \leq X} e(qf(p)) = \sum_{X_1 < n \leq X} \Lambda_1(n) e(qf(n)) \\
         =&\sum_{n=X_1}^{X-1} (e(qf(n)) - e(qf(n+1))) (\pi(n) - \pi(X_1)) + e(qf(X)) (\pi(X) - \pi(X_1))  \notag \\
         =&\sum_{n=X_1}^{X-1} [e(qf(n)) - e(qf(n+1))] [Li (n) - Li (X_1)] + e(qf(X))[Li (X) - Li (X_1)]  \notag \\
         &+ O \left( \left(X \exp (- C \sqrt{\log X}) \right) \left( \sum_{n=X_1}^{X-1}  |e(qf(n)) - e(qf(n+1))| +1 \right)\right). \notag
\end{align}
Since
$$ |e(qf(n)) - e(qf(n+1))| = | 1 - e(q(f(n+1) -f(n))) | \ll q|f'(n)| \leq \frac{q \log^{15}X}{X},$$
the $O$-term is 
$$
 \ll X \exp(- C \sqrt{\log X}) \left(\sum_{n \leq X}\frac{q \log^{15}X}{X} +1 \right) \ll \frac{\pi(X)}{\log X}. 
$$
Now we use Lemma \ref{lem5} with $a_n = Li(n) - Li(X_1)$ and $b_n = e(qf(n)) - e(qf(n+1))$ and obtain
 \begin{equation}
\label{Con}
\begin{split}
 S_0 &= \sum_{n=X_1}^X \int_n^{n+1} \frac{dx}{\log x} e(qf(n)) + O \left( \frac{\pi(X)}{\log X}\right) \\ 
&= \sum_{n=X_1}^X \frac{1}{\log n} e(qf(n))  + O\left(\frac{\pi(X)}{\log X}\right).     
\end{split}
\end{equation}
From \eqref{eqn3.1}, $\left| \frac{d}{dn} (q f(n))\right| = |qf'(n)| \leq \frac{1}{2}$\footnote{Throughout this section we are tacitly assuming that $n$ is a continuous variable.}, so $\| qf'\| > q f'(X)$ for $Y \leq n \leq X$.
Use Lemma \ref{lem5} with $a_n = \frac{1}{\log n}$ and $b_n = e(qf(n))$ and then use Lemma \ref{lem3} and \eqref{eqn3.1}: 
\begin{equation}
\label{eqn2}
S_0 \ll  \frac{1}{\log X} \frac{1}{q f'(X)} + \frac{\pi(X)}{Q} \ll \frac{\pi(X)}{Q}. 
\end{equation}

Now we need to evaluate
\begin{equation*}
S_1 := \sum_{X_0 < p \leq X_1} e(q f(p)) = \sum_{X_0 < n \leq X_1} \frac{\Lambda(n)}{\log n} e(qf(n)) + O(\sqrt{X}). 
\end{equation*}
To complete the proof we have to show that $S_1 \ll \frac{\pi(X)}{Q}$. 
Using Lemma \ref{lem5} with $a_n = \frac{1}{\log n}$ and $b_n = \Lambda(n) e(qf(n))$ we obtain
\begin{equation*}
|S_1|  \ll \frac{1}{\log X} \max_{X_2 \leq X_1} \left| \sum_{ X_0 < n \leq  X_2} \Lambda(n) e(q f(n))\right|  + \sqrt{X}.
\end{equation*}
To estimate $\sum\limits_{X_0 < n \leq X_2} \Lambda(n) e(qf(n))$, 
we divide the interval $( X_0, X_2 ]$ into $\ll \log Q$ subintervals of the form $(X_3, X_4] \subset (X_3, 2 X_3]$ and evaluate 
$$S_2 := \sum_{X_3 < n \leq X_4} \Lambda(n) e (q f(n)).$$
Using Lemma \ref{lem7} with $g(n) = e(qf(n)) \cdot 1_{(X_3, X_4]} (n)$ and some $u, v$ to be defined later, we obtain
\begin{equation}
 \label{eqn4} 
 |S_2| \leq |S_3| + |S_4| + |S_5|
\end{equation}
where $S_3, S_4,$ and $S_5$ correspond to the sums $T_1, T_2$ and $T_3$ in Lemma \ref{lem7}:
\begin{align*}
S_3 &= \sum_{d \leq u} \sum_{\frac{X_3}{d} \leq m \leq \frac{X_4}{d}} \mu(d) \, (\log m) \, e(qf(dm)),\\
S_4 &= \sum_{m \leq uv} \sum_{\frac{X_3}{m} \leq r \leq \frac{X_4}{m}} a(m) \, e(qf(mr)), \quad a(m) = \sum\limits_{d \leq u} \sum\limits_{\substack{n \leq v \\ dn=m}} \mu(d) \, \Lambda(n)\\
S_5 &= \sum_{m > u} \sum_{\substack{\frac{X_3}{m} \leq n \leq \frac{X_4}{m} \\ v > n} } b(m) \, \Lambda(n) \, e(qf(mn)), \quad b(m) = \sum\limits_{\substack{d \leq u \\ d | m}} \mu(d). 
\end{align*}
To estimate $S_3, S_4$ and $S_5$, we will consider two cases: 
$$(a)\, |f'(X_3)| \leq \log^{-10} X \quad \text{and} \quad  (b)\, |f'(X_3)| \geq \log^{-10} X.$$

%%%%%
Case (a). Assume first that $|f'(X_3)| \leq \log^{-10} X$. We take $u=v= \log^3 X$. 

Let us estimate the sum $S_3$. For $X_3 \leq dm \leq X_4$, we have
$$\left| \frac{\partial}{\partial m} (q f(dm))\right| = \left| qd \, f'(dm) \right| \leq \frac{qd}{\log^{10} X} \leq \frac{1}{2} \, \, \, \text{and}
\, \, \, \left| \frac{\partial}{\partial m} (q f(dm))\right| =  qd \left| f'(dm) \right| \geq qd \, \frac{\log^{15} X}{X}.$$
Using Lemma \ref{lem5} first and Lemma \ref{lem3} after that, we obtain
\begin{align}
\label{eqn5'}
|S_3| &\leq \sum_{d \leq u} \left| \sum_{\frac{X_3}{d} \leq m \leq \frac{X_4}{d}} \log m \, e (q f(dm)) \right| \\
\label{eqn5} 
&\ll \log X \sum_{d\leq u} \frac{X}{qd \log^{15} X } \leq \frac{X}{ \log^3 X}.
\end{align}

Let us estimate the sum $S_4$. For $X_3 \leq mr \leq X_4$, we have
$$ \left| \frac{\partial }{\partial r} (q f(mr)) \right| = qm |f'(mr)| \ll \frac{qm}{\log^{10} X} \leq \frac{1}{2} \, \, \, \text{and} \, \, \,
 \left| \frac{\partial }{\partial r} (q f(mr)) \right| \geq qm \frac{\log^{15}X}{X}.$$

Using Lemma \ref{lem3} to evaluate the sum over $r$, we get 
\begin{align}
\label{eqn6'}
|S_4| &\leq \sum_{m \leq uv} |a(m)| \, \left| \sum_{\frac{X_3}{m} < r \leq \frac{X_4}{m}} e (q f(mr)) \right| \\
         \label{eqn6}
         & \ll \sum_{m \leq uv} |a(m)| \, \frac{X}{qm \log^{15}X} 
          \ll \frac{X}{ \log^{15}X} \sum_{d \leq u} \frac{|\mu(d)|}{d} \sum_{n \leq v} \frac{\Lambda(n)}{n} \ll \frac{X}{\log^3 X}.
\end{align}

To evaluate the sum $S_5$, we divide the interval $\left( u, \frac{X_4}{u} \right]$ into $\ll \log X$ subintervals of the form $(M, M_1] \subset (M, 2M]$ and evaluate
$$S_6 := \sum_{M < m \leq M_1} \sum_{\substack{\frac{X_3}{m} \leq n \leq \frac{X_4}{m} \\ v > n} } b(m),$$
Note that 
\begin{equation}
\label{S_5}
|S_5| \ll \log X \cdot \max |S_6|.
\end{equation}

To estimate $S_6$, we will consider two cases $M \geq \sqrt{X_4}$ and $M < \sqrt{X_4}$.

If $M \geq \sqrt{X_4}$, we use the Cauchy-Schwartz inequality to obtain
\begin{align}
\label{eqn7}
\left| S_6 \right|^2 &= \left| \sum_{M < m \leq M_1} b(m)  \sum_{\substack{ \frac{X_3}{m} < n \leq \frac{X_4}{m} \\ n > v} } \Lambda(n) e(q f(mn)) \right|^2 \notag \\
&\leq \left(\sum_{M < m \leq M_1} |b(m)|^2 \right) \left( \sum_{M < m \leq M_1} \left| \sum_{\substack{\frac{X_3}{m} < n \leq \frac{X_4}{m}\\ n >v }} \Lambda(n) e(q f(mn)) \right|^2 \right)
\end{align}
With the help of Remark \ref{etc}, the first term in \eqref{eqn7} can be estimated as follows:
 $$\sum_{M < m \leq M_1} |b(m)|^2 \ll M \log^3 u.$$
To estimate the second term in \eqref{eqn7}, 
we will apply Lemma \ref{lem6} with the largest integer $H \leq \frac{X_3}{M_1}$ such that 
\begin{equation}
\label{eqn9}
H q \max_{X_3 \leq t \leq X_4} |f'(t) + t f''(t)| \leq \frac{1}{2}. 
\end{equation} 
Then,
\begin{align}
\label{eqn10}
 |S_6|^2 &\ll M \log^3 u \, |A+ B|, \,\, \text{where}  \\ 
 A &=\sum_{M < m \leq M_1} \frac{X_4}{MH} \left( \sum_{h=1}^H \sum_{\substack{\frac{X_3}{m} \leq n, n+h \leq \frac{X_4}{m} \\ n, n+h > v}} \Lambda(n) \Lambda(n+h) \, e (q f(mn+mh) - q f(mn)) \right) \notag \\
 B &= \sum_{M < m \leq M_1} \frac{X_4}{MH} \sum_{\substack{\frac{X_3}{m} \leq n \leq \frac{X_4}{m} \\ n > v}} \Lambda^2(n). \notag
\end{align}
Note that for some $t_0 \in [t,t_1] = [mn, mn+mh]$
\begin{align*}
&\left| \frac{\partial}{\partial m} (qf(mn+mh) - qf(mn)) \right| = \big| q [(n+h) f'(mn+mh)- n f'(mn) ] \big| \\ 
&= \frac{q}{m} |(t_1 f'(t_1) - t f'(t))| = \frac{q(t_1 -t)}{m} \left| \frac{d}{dt} (t f'(t))|_{t=t_0} \right| = qh |f'(t_0) + t_0 f''(t_0)|.
\end{align*}
For any pair $(h,n)$, where $n, n+h \in (\frac{X_3}{M_1}, \frac{X_4}{M}]$, 
we will use Lemma \ref{lem3} with $\lambda= q h \min\limits_{X_3 \leq t \leq X_4} |f'(t) + t f''(t)|$ to estimate
$$\left| \sum_{M < m \leq M_1} e (qf(mn+mh) - qf(mn))\right| \ll \frac{1}{\lambda}.$$
To estimate \eqref{eqn10}, we consider two cases $H = \frac{X_3}{M_1}$ and $H < \frac{X_3}{M_1}$ for \eqref{eqn9}.
If $H = \frac{X_3}{M_1},$ use $f'(n) > \frac{\log^{15}X}{X}$ for $n \geq X_0$ and \eqref{e3} from Proposition \ref{Fejer} to get 
\begin{equation}
\label{^}
\lambda = qh |f'(t_0)| \left| 1 + \frac{t_0 f''(t_0)}{f'(t_0)}\right| \gg qh \frac{\log^{15} X}{X} \frac{1}{\log^2 X}.
\end{equation}
Hence
\begin{align}
\label{fin S_6_1}
|S_6|^2 &\ll  M \log^3 u \sum_{h=1}^H  
\sum_{\frac{X_3}{M_1} < n, n+h \leq \frac{X_4}{M}} \Lambda(n) \Lambda(n+h) \frac{X}{qh \log^{13}X} + M  \log^3 u \, \cdot  \frac{X_3}{M} \,  \log \frac{X_3}{M} \notag \\
&\ll \frac{X X_3 \log^3 u}{\log^{11} X} +  X \log^3 u \, \log X \ll \frac{X^2}{\log^6 X}. 
\end{align}
For $H < \frac{X_3}{M_1}$, note that 
$$|f'(X_4)| \gg \frac{ |f'(X_3)|}{\log X} \gg \frac{1}{\log X} \max\limits_{X_3 \leq t \leq X_4} |t f''(t) + f'(t)|.$$ 
From \eqref{e3}, \eqref{e1} and the fact that $(H+1) \, q \max\limits_{X_3 \leq t \leq X_4} |t f''(t) + f'(t)| \geq \frac{1}{2}$, we get
\begin{equation}
\begin{split}
\label{^2}
\lambda &= qh \min_{X_3 \leq t \leq X_4 } |tf''(t) + f'(t)| \gg qh \min_{X_3 \leq t \leq X_4} \frac{|f'(t)|}{\log^2 t} \gg \frac{qh}{\log^2 X} |f'(X_4)| \\
             &\gg \frac{qh}{\log^3 X} \max_{X_3 \leq t \leq X_4} |t f''(t) + f'(t)| \gg \frac{h}{2H \log^3 X}.
\end{split}
\end{equation} 
Using Lemma \ref{lem3}, we can estimate $A$:
\begin{align*}
A &\ll \frac{X_4}{MH} \sum_{h=1}^H \sum_{\frac{X_3}{M_1} < n, n+h \leq \frac{X_4}{M}} \Lambda(n) \Lambda(n+h) \frac{2H \log^3 X}{h}  \\
&\ll \frac{X_4}{MH} \left( \frac{X_3}{M_1} \log \frac{X_3}{M_1} \right) \left( 2H \log^3 X \right) \log H \ll \frac{X^2 \log^5 X }{M^2}.
\end{align*}
Now let us estimate $B$. From \eqref{e3}, $\lambda = qh \min\limits_{X_3 \leq t \leq X_4} |f'(t) + t f''(t)| \leq qh |f'(X_3)|$. So from \eqref{^2},
\begin{equation}
\label{^3} 
\frac{1}{H} \ll \frac{2 \lambda \log^3 X}{h} \ll \frac{2 \log^3 X}{h} qh |f'(X_3)| \ll \frac{2q \log^3 X}{\log^{10} X} \leq \frac{2}{\log^6 X}.
\end{equation}
Then,
$$
B \ll \sum_{M < m \leq M_1} \frac{X_4}{MH} \left( \frac{X_3}{m} \log \frac{X_3}{m} \right) 
\ll \frac{X^2 \log X}{MH} \ll \frac{X^2}{M \log^5 X}.$$
Therefore, we get
\begin{equation}
\label{fin S_6_2}
\begin{split}
|S_6|^2 &\ll \frac{X^2 \log^3 u \, \cdot \log^5 X }{M} + \frac{X^2 \log^3 u}{\log^5 X}   \\
&\ll X^{\frac{3}{2}} \log^{6} X + \frac{X^2}{\log^{9/2}X} \ll \frac{X^2}{\log^{9/2} X}.
\end{split}
\end{equation}
From \eqref{fin S_6_1} and \eqref{fin S_6_2}, summing over all $M \geq \sqrt{X_4}$, we obtain
\begin{equation}
\label{eqn13}
|S_5| \ll \log X \cdot \frac{X}{\log^{9/4} X} =  \frac{X}{\log^{5/4}X }.
\end{equation}

If $M < \sqrt{X_4}$, we interchange the order of summation in \eqref{eqn7} and again use the Cauchy-Schwartz inequality:
\begin{align*}
|S_6|^2
&\leq \left| \sum_{\substack{N \leq n \leq N_1 \\ n > v}} \Lambda(n) \sum_{\frac{X_3}{n} \leq m \leq \frac{X_4}{n}} b(m) e(qf(mn)) \right|^2 \\
&\leq \left( \sum_{N \leq n \leq N_1} \Lambda^2(n) \right) \left( \sum_{N \leq n \leq N_1} \left| \sum_{\frac{X_3}{n} \leq m \leq \frac{X_4}{n}} b(m) e(qf(mn)) \right|^2 \right),
\end{align*}
where $N =\frac{X_3}{M_1}$ and $N_1 = \frac{X_4}{M}$. The rest of the proof is similar to proof for the case $M \geq \sqrt{X_4}$. 

Let $H$ be the largest integer $\leq M/2$ such that
$qH \max\limits_{X_3 \leq t \leq X_4} |f'(t) + t f''(t) | \leq 1/2.$ Then,
\begin{align}
\label{eqn14}
|S_6|^2 &\ll  N \log N \, |C+D|, \,\, \text{where} \\ 
C &= \sum_{N \leq n \leq N_1} \frac{X_4}{NH} \left(  \sum_{h=1}^H \sum_{\frac{X_3}{n} \leq m, m+h \leq \frac{X_4}{n}}  b(m) b(m+h) e(qf(mn+mh) - qf(mn)) \right) \notag \\
D &= \sum_{N \leq n \leq N_1} \frac{X_4}{NH} \sum_{\frac{X_3}{n} \leq m \leq \frac{X_4}{n}} |b(m)|^2. \notag
\end{align}
For any pair $(m,h)$, where $m, m+h \in (\frac{X_3}{N_1}, \frac{X_4}{N}]$,  using Lemma \ref{lem3} with 
$\lambda = qh \min\limits_{X_3 \leq t \leq X_4} |f'(t) + tf''(t)| \leq 1/2$, we get
$$ \left| \sum_{N \leq n \leq N_1} e(qf(mn+mh) - qf(mn))\right| \ll \frac{1}{\lambda}.$$
From Remark \ref{etc},  $$\sum\limits_{M \leq m \leq M_1} |b(m)|^2 \ll M \log^3 u, \quad \sum_{M \leq m, m+h \leq M_1} |b(m)| \, |b(m+h)| \ll M \log^3 u.$$
If $H = \frac{X_3}{N_1} (=M/2)$, then, as above (see \eqref{^}), we get $\lambda \geq \frac{qh \log^{13}X}{X}$ and so  
\begin{align}
\label{fin S_6_3}
|S_6|^2 &\ll N \log N \left(  \sum_{h=1}^H \sum_{M \leq m, m+h \leq M_1} |b(m)| \, |b(m+h)| \frac{X}{qh\log^{13}X} + \sum_{M \leq m \leq M_1} |b(m)|^2 \right) \notag \\
 &\ll \frac{X_4 X \log^3 u}{\log^{11}X} + X_4 \log \frac{X_4}{M} \log^3 u \ll \frac{X^2}{\log^6 X}.
\end{align}
If $H< M/2$, then, as in \eqref{^2} and \eqref{^3}, we get $\lambda \geq \frac{h}{2H \log^3 X}$ and $\frac{1}{H} \ll \frac{1}{\log^6 X}$. Thus we obtain
$$C \ll \frac{X_4}{NH} (M \log^3 u) \, (2H \log^3 X) \log H \ll \frac{X^2}{N^2} \, \log^5 X $$
and $$D \ll \frac{X_4}{NH} \sum_{N \leq n \leq N_1} \frac{X_3}{n} \log^3 u \ll \frac{X_4}{NH} X_3 \log^3 u \ll \frac{X^2 \log^3 u}{N \log^6 X},$$
so 
\begin{equation}
\label{fin S_6_4}
|S_6|^2 \leq \frac{X^2  \log^6 X }{N} + \frac {X^2 \log^3 u}{\log^5 X} \ll \frac{X^2}{\log^{9/2}X}.
\end{equation}

From \eqref{fin S_6_3} and \eqref{fin S_6_4}, summing $|S_6|$ over all $M < \sqrt{X_4}$ gives again the same estimate for $S_5$ as in \eqref{eqn13}. 
%%%%%%%%%%%%
Therefore, 
$$|S_1| \ll \frac{\log Q}{\log X} \left(\frac{X}{\log^3 X} + \frac{X}{\log^3 X} + \frac{X}{\log^{5/4}X} \right) + \sqrt{X} \ll \frac{\pi(X)}{Q}.$$

%%%%%%%%%%%%

Case (b). Now we assume that $|f'(X_3)| \geq \log^{-10} X$. 
Take $u=v= X^{1/10}$ and argue as above but use Lemma \ref{lem2.3.1} with $k=1$ instead of Lemma \ref{lem3}. 
The formulas \eqref{e0} and \eqref{e3} imply that for $x \in (X_3, X_4] \subset (X_3, 2 X_3]$ 
\begin{equation}
\label{estimate f''}
|f''(x)| \ll \frac{|f'(x)|}{x} \ll \frac{1}{X_3}, \quad |f''(x)| \gg  \frac{|f'(x)|}{x \log^2 x} \gg \frac{1}{X_3 \log^{13} X_3} . 
\end{equation}

To evaluate $S_3$, let $\lambda_2 = \left| \frac{d^2}{dm^2} (qf(dm)) \right| = qd^2 |f''(dm)|$. 
Then 
$$\lambda:= \frac{qd^2}{X_3 \log^{13} X_3} \ll \lambda_2 \ll \frac{qd^2}{X_3}=: \alpha \lambda.$$
Apply \eqref{nov18} in Lemma \ref{lem2.3.1} with $k=1$ to evaluate the sum in \eqref{eqn5'}:
\begin{align*}
|S_3| &\ll \sum_{d \leq u} \log X \left| \sum_{\frac{X_3}{d} \leq m \leq \frac{X_4}{d}} e(qf(dm))\right| \\
         &\ll \log X \sum_{d \leq u} \frac{X_3}{d} \left( \sqrt{\frac{qd^2}{X_3}} + \sqrt{\frac{d \log^{14}X_3}{X_3}}\, \right) \ll X^{4/5}.
\end{align*}

Similarly, to evaluate $S_4$ let $\mu_2 = \left|\frac{d^2}{dr^2} qf(mr)\right| = qm^2 |f''(mr)|$. 
Then 
$$\lambda:= \frac{qm^2}{X_3 \log^{13} X_3} \ll \mu_2 \ll \frac{qm^2}{X_3}=: \alpha \lambda.$$
Apply Lemma \ref{lem2.3.1} with $k=1$ to evaluate the sum in \eqref{eqn6'}:
\begin{align*}
|S_4| &\ll \sum_{m \leq uv} |a(m)| \left| \sum_{\frac{X_3}{m} \leq r \leq \frac{X_4}{m}} e(qf(mr))\right| \\
         &\ll  \sum_{m \leq uv} |a(m)| \frac{X_3}{m} \left( \sqrt{\frac{qm^2}{X_3}} + \sqrt{\frac{m \log^{14}X_3}{X_3}}\,\right) \ll X^{4/5}.
\end{align*}

To evaluate $S_6$ for $M \geq \sqrt{X_4}$, we evaluate the sum over $m$ in \eqref{eqn10} with $H = \min \{\frac{X_3}{M_1}, M^{1/3}\}$. 
Denoting $g(m) = q f(mn+ mh) - qf(mn)$ with $(t_1 = mn+mh, t= mn, t_0 \in [t,t_1])$, as above, we get 
\begin{align*}
g''(m) 
&= q [(n+h)^2 f''(mn + mh) - n^2 f(mn)] = \frac{q}{m^2} [t_1^2 f''(t_1) - t^2 f''(t)] \\
&= \frac{q}{m^2} \cdot  mh \cdot \frac{d}{dt} (t^2 f''(t)) \big|_{t= t_0} = \frac{qh}{m} (2 t_0 f''(t_0) + t_0^2 f'''(t_0)).
\end{align*}  
Using \eqref{e4} from Proposition \ref{Fejer} we get
$$
\left| g'' (m) \right| 
= \frac{qh}{m} |t_0 f''(t_0)| \left|2+ \frac{t_0 f'''(t_0)}{ f''(t_0)} \right| \gg \frac{qh}{M} \frac{|f'(t_0)|}{\log^2 X} \frac{1}{\log^2 X} \gg \frac{qh}{M \log^{13} X} =: \lambda
$$
and  
$$ \left| g'' (m) \right| \ll \frac{qh}{M} =: \alpha \lambda.$$
Apply Lemma \ref{lem2.3.1} with $k=1$ to the sum in \eqref{eqn10}:
\begin{align*}
A &\ll \frac{X_4}{MH} \sum_{h=1}^H \sum_{\frac{X_3}{M_1} \leq n, n+h \leq \frac{X_4}{M}} \Lambda(n) \Lambda(n+h) \left( M \sqrt{\frac{qh}{M}} +  M \sqrt{\frac{\log^{15} X \log M}{M}}\right) \\
&\ll \frac{X^2}{M} \log X \left( \sqrt{\frac{qH}{M}} + \sqrt{\frac{\log^{15} X \log M}{M}}\right) \ll \frac{X^2 \log^9 X}{M^{4/3}}.
\end{align*}
Also, $ B \ll \frac{X_3^2}{MH} \log X$. If $H = \frac{X_3}{M_1} \leq M^{1/3}$, then $1/H =\frac{M_1}{X_3} \leq \frac{1}{X_3} \frac{X_4}{u} \ll X^{-1/10}$ and if $H = M^{1/3}$, then $1/H \ll X^{-1/6}.$ So $ B \ll \frac{1}{M} X^{19/10} $. 
Then $$|S_6|^2 \ll M \log^3 X (A+B) \ll \frac{X^2 \log^{12} X}{M^{1/3}} + X^{19/10} \log^3 X  \ll \frac{X^2}{M^{1/6}}+ X^{19/10} \leq X^{23/12}.$$

If $M < \sqrt{X_4}$ we interchange the order of summation in \eqref{eqn7} and evaluate the sum over $n$ in \eqref{eqn14} using Lemma \ref{lem2.3.1} with $k=1$,  in exactly the same way as above, we obtain $|S_6|^2 \ll X_3^2  / N^{1/6} \leq X^{23/12}$.

Thus, we have $S_5 \ll X^{23/24} \log X.$ 
Combining this with $S_3 \ll X_4^{4/5}$ and $S_4 \ll X_4^{4/5}$, we get
$$S_2 \ll X^{4/5} + X^{23/24} \log X$$
and we obtain
$$S_1 \ll \frac{\log Q}{\log X} (X^{4/5} + X^{23/24} \log X ) + \sqrt{X} \leq \frac{\pi(X)}{Q}.$$ 
This proves Theorem \ref{thm1}.
\end{proof}

%%%%%%%%%%%%%%%%%%%%%%%%%%%%%%%%%%%%%%%%%%%%%%%%%%%%%%%%%
\begin{Theorem}
\label{thm2}
Let $f(x) \in {\bf U}$ be of the type $x^{l+}$ for some $l \geq 1$.
Then $(f(p))_{p \in \mathcal{P}}$ is u.d. mod 1.
\end{Theorem}

\begin{proof}
By Lemma \ref{lemmaS}, it is enough to show that $Q(X) = \log X$ satisfies
\begin{equation*}
S := \sum_{X_0 < p \leq X} e(qf(p)) \ll \frac{\pi(X)}{Q},
\end{equation*}
where $X_0 = \frac{X}{Q}$.
Note that, by Lemma \ref{lem5}, 
$$S  = \sum_{X_0 < n \leq X} \frac{\Lambda(n)}{\log n} e(qf(n)) + O({\sqrt{X}}) \ll \frac{1}{\log X} \max_{X_2 \leq X} \left| \sum_{X_0 < n \leq X_2} \Lambda(n) e(qf(n)) \right| + \sqrt{X}.$$

\noindent Now let us estimate $\sum\limits_{X_0< n \leq X_2} \Lambda(n) e(qf(n))$. 
Since $X_0 = \frac{X}{Q}$ and $X_2 \leq X$, 
we can divide the interval $(X_0, X_2]$ into $\ll \log Q$ subintervals of the form $(X_3,  X_4] \subset (X_3, 2X_3]$ 
and evaluate the corresponding sum 
$$S_2 := \sum\limits_{X_3 < n \leq X_4} \Lambda(n) e(qf(n)).$$ 
By Lemma \ref{lem7} with $u=v= X^{\epsilon}$ for some small $\epsilon >0$, 
we have $|S_2| \leq |S_3| + |S_4| + |S_5|$, where $S_3, S_4$ and $S_5$ correspond to $T_1, T_2$ and $T_3$ in Lemma \ref{lem7}:
\begin{align}
\label{eqS3}
S_3 &= \sum_{d \leq u} \sum_{\frac{X_3}{d} \leq m \leq \frac{X_4}{d}} \mu(d) (\log m) e(q f(dm)), \\
\label{seS4}
S_4 &= \sum_{m \leq uv} \sum_{\frac{X_3}{m} \leq r \leq \frac{X_4}{m}} a(m) e(q f(mr)), \quad a(m) = \sum\limits_{d \leq u} \sum\limits_{\substack{n \leq v \\ dn=m}} \mu(d) \Lambda(n)\\
\label{eqS5}
S_5 &= \sum_{m > u} \sum_{\substack{\frac{X_3}{m} \leq n \leq \frac{X_4}{m} \\ v > n} } b(m) \Lambda(n) e(qf(mn)), \quad b(m) = \sum\limits_{\substack{d \leq u \\ d | m}} \mu(d). 
\end{align}
We will show that $|S_i| \ll X^{1- \epsilon_0}$ for some $\epsilon_0 > 0$, which implies that 
$$|S| \ll \frac{\log Q}{\log X} X^{1- \epsilon_0} \leq \frac{\pi(X)}{Q}.$$

\noindent Let us estimate $S_3$ first.
Let $\lambda_j := \left| \frac{\partial^{l+j}}{\partial m^{l+j}} q f (dm) \right| = qd^{l+j} |f^{(l+j)} (dm)|$. From Proposition \ref{tempered}, for $dm \in (X_3, X_4]$,
\begin{equation}
\label{lambda1}
qd^{l+j} X_3^{\beta-l-j-\epsilon} \ll \lambda_j \ll qd^{l+j} X_3^{\beta - l- j + \epsilon},
\end{equation}
where $\beta = \lim\limits_{x \rightarrow \infty} \frac{\log f(x)}{\log x} \in [l,l+1]$.
Choose $j$ such that $qd^{l+j} X_3^{\beta - j - l + \epsilon} \ll X_3^{-\epsilon}$. 
We will use Lemma \ref{lem2.3.1} with $k=l+j-1$, $\lambda = qd^{k+1} X_3^{\beta -k -1 -\epsilon}$ and $\alpha \lambda = qd^{k+1} X_3^{\beta - k-1 + \epsilon}$.  
Note that 
$$\alpha \lambda \ll X_3^{-\epsilon}, \quad \lambda \left( \frac{X_3}{d}\right)^{k+1} \gg X_3^{\epsilon} \quad \text{and} \quad \frac{\alpha \log^k X_3/d}{X_3/d} \ll X_3^{-\epsilon}.$$
Using Lemmas \ref{lem5} and \ref{lem2.3.1} we get: 
$$|S_3| \ll \sum_{d \leq u}  \log X_3 \frac{X_3}{d} X_3^{-\frac{\epsilon}{2K}} \ll X^{1 - \frac{\epsilon}{4K}}.$$

\noindent To evaluate $S_4$, we will use an argument similar to that utilized in evaluating $S_3$. Denote 
$$\mu_j := \left| \frac{\partial^{l+j}}{\partial r^{l+j}} (qf(mr)) \right| = qm^{l+j} |f^{(l+j)} (mr)|.$$ 
Then, similar to \eqref{lambda1},
\begin{equation}
 \label{mu1}
 qm^{l+j} X_3^{\beta - l - j  - \epsilon} \ll \mu_j \ll qm^{l+j} X_3^{\beta - l - j + \epsilon},
\end{equation}
so by chooising $j$ such that  $qm^{l+j} X_3^{\beta - j - l + \epsilon} \ll X_3^{-\epsilon}$ and $k = l+j-1$,
$$|S_4| \ll \sum_{m \leq uv} |a(m)| \frac{X_3}{m} X_3^{- \frac{\epsilon}{2K}} \ll \sum_{m \leq u v} \frac{|a(m)|}{m}  X_3^{1- \frac{\epsilon}{2K}} \ll X^{1- \frac{\epsilon}{4K}}.$$

\noindent In order to evaluate $S_5$, we will estimate $S_6$ as in the proof of Theorem \ref{thm1}. 
So we need to evaluate the sum in \eqref{eqn10} for $M \geq \sqrt{X_4}$ and the sum in \eqref{eqn14} for $M < \sqrt{X_4}$.

\noindent For the sum in \eqref{eqn10}, take $H =\frac{1}{2}  X_3^{{\epsilon}}$. 
Note that $\frac{X_3}{M_1} \geq \frac{X_4}{2 M_1} \geq \frac{1}{2}u \geq H$. 
Then
$$B \ll \sum_{M < m \leq M_1} \frac{X_3}{MH} \frac{X_3}{m} \log X_3 \ll \frac{X_3^2 \log X_3}{MH} \ll \frac{X_3^{2- \epsilon} \log X_3}{M} \ll \frac{X^{2- \epsilon} \log X}{M}.$$
To evaluate $A$, we need to estimate 
\begin{equation}
\label{evA}
\left| \sum\limits_{M < m \leq M_1} e(qf(mn+mh) - qf(mn)) \right|,
\end{equation}
where $n, n+h \in (\frac{X_3}{M_1}, \frac{X_4}{M}]$.
Using Proposition \ref{tempered}, we get
\begin{align}
\label{est lambda1j}
\lambda_{1,j} &:= \left| \frac{\partial^{l+j}}{\partial m^{l+j}} (qf(mn+mh) - qf(mn))\right| \notag \\
&= \frac{qh}{m^{l+j-1}} t_0^{l+j-1} \left| (l+j)f^{(l+j)}(t_0) + t_0 f^{(l+j+1)}(t_0) \right| \notag \\
&= \frac{qh}{m^{l+j-1}} t_0^{l+j-1} \left| f^{(l+j)} (t_0) \right|  \left| l+j + \frac{t_0 f^{(l+j+1)} (t_0)}{f^{(l+j)} (t_0)} \right|
\end{align}
for some $t_0 \in [mn, mn+mh]$, so
\begin{equation}
\label{lambda2}
\frac{qh}{M^{l+j-1}} X_3^{\beta- 1 -\epsilon} \ll \lambda_{1,j} \ll \frac{qh}{M^{l+j-1}} X_3^{\beta-1+ \epsilon}.
\end{equation}
Fix an integer $j$ such that $\frac{qh}{M^{l+j-1}} X_3^{\beta-1+ \epsilon} \ll X_3^{-\epsilon}$. 
We will evaluate \eqref{evA} via Lemma \ref{lem2.3.1} with $k=l+j-1$ 
and $\lambda =\frac{qh}{M^{l+j-1}} X_3^{\beta- 1 -\epsilon} $ and $\alpha \lambda = \frac{qh}{M^{l+j-1}} X_3^{\beta-1+ \epsilon}$. 
Note that $\alpha \lambda \leq X_3^{-\epsilon}$, $\lambda M^{k+1} \gg X_3^{\epsilon}$ and $\frac{\alpha \log^k M}{M} \ll X_3^{- \epsilon}$. 
Then, 
$$A \ll \frac{X_4}{MH} \sum_{h=1}^H \sum_{\frac{X_3}{M_1} \leq n , n+h \leq \frac{X_4}{M}} \Lambda(n) \Lambda(n+h) M X_3^{-\epsilon/2K} \ll \frac{X^{2 - \epsilon/2K} \log X}{M}.$$
Thus $|S_6|^2 \ll X^{2-\epsilon/2K} \log^4 X$.

For the sum in \eqref{eqn14} for $M < \sqrt{X_4}$, let  $N= \frac{X_3}{M_1}$ and $N_1 = \frac{X_4}{M}$. Take $H = \frac{1}{2} X_3^{\epsilon}$ so that 
$$D \ll \sum_{N < n \leq N_1} \frac{X_4}{NH} \frac{X_3}{n} \log^3 u \ll \frac{X_3^{2 - \epsilon} \log^3 X_3}{N} \ll \frac{X^{2-\epsilon} \log^3 X}{N}.$$
Denote
\begin{align}
\label{est mu1j}
\mu_{1,j} &:= \left| \frac{\partial^{l+j}}{\partial n^{l+j}} (qf(mn+mh) - qf(mn)) \right|
=  qm^{l+j} \cdot mh \, \left|f^{(l+j+1)} (t_0) \right| \notag \\
&= \frac{qh}{n^{l+j+1}} t^{l+j+1} \left| f^{(l+j+1)} (t_0) \right|,
\end{align}
for some $t_0 \in [mn, mn+mh]$. Thus,
\begin{equation}
\label{mu2}
\frac{qh}{N^{l+j+1}}X_3^{\beta -\epsilon} \ll \mu_{1,j} \ll \frac{qh}{N^{l+j+1}} X_3^{\beta +\epsilon}.
\end{equation}
Fix $j$ such that $\frac{qh}{N^{l+j+1}} X_3^{\beta + \epsilon} \ll X_3^{-\epsilon}$. 
Applying Lemma \ref{lem2.3.1} with $k=l+j -1$ as above to evaluate $A$, we also get 
$$C \ll \frac{X^{2-\epsilon/2K} \log X}{N}.$$
Thus $|S_6|^2 \ll X^{2-\epsilon/2K} \log^4 X$, so $|S_5| \ll X^{1- \epsilon/8K}$.
\end{proof}

%%%%%%%%%%%%%%%%%%%%%%%%%%%%%%%%%%%

\begin{Theorem}
\label{thm3}
Let $f_1(x) \in {\bf U}$ is a subpolynomial function such that 
\begin{itemize}
\item $f_1$ is of the type $x^{0+}$ and $\lim\limits_{x \rightarrow \infty} \left| \frac{f_1(x)}{\log x}\right| = \infty$ or
\item $f_1$ is of the type $x^{l+}$ for some $l \geq 1$.
\end{itemize}
Let $f(x) = f_1(x) + f_2(x)$, where $f_2(x) \in \mathbb{R}[x]$ is a polynomial. Then $(f(p))_{p \in \mathcal{P}}$ is u.d. $\bmod \, 1$.
\end{Theorem}

\begin{proof} 
We assume that $f_2(0)=0$ and write $f_2(x) = \alpha_1 x + \cdots + \alpha_t x^t \in \mathbb{R}[x]$.
If $t \leq l$, then $f(x) \in {\bf U}$ is of the type $x^{l+}$, so the result follows from Theorems \ref{thm1} and \ref{thm2}. 
So we also assume that $t \geq l+1$. 

Let us consider first $l \geq 1$. 
Note that for $j \geq t - l + 1$, $f_2^{(j+l)}(x) =0$. 
Moreover by choosing $j$ appropriately as in the proof of Theorem \ref{thm2} with this additional condition, 
we get the same estimates for $\lambda_j, \mu_j, \lambda_{1,j}, \mu_{1,j}$ as in \eqref{lambda1}, \eqref{mu1}, \eqref{lambda2} and \eqref{mu2}. 
(Here $\beta = \lim_{x \rightarrow \infty} \frac{\log f_1(x)}{\log x}$.) 
Hence, we can evaluate $S_3, S_4$ and $S_5$ as in the proof of Theorem \ref{thm2}, which concludes the proof for $l \geq 1$.

Now consider $l=0$. 
We will consider two cases:

\noindent Case I. Suppose that $|f_1'(x)| \geq \frac{\log^c x}{x}$ eventually, where $c \geq 9 \cdot 2^t + 3t +4$. 
As in the proof of Theorem \ref{thm2}, with $u=v= X^{\epsilon}$ for some small $\epsilon>0$, 
we will show that $S_3, S_4$ and $S_5$ in \eqref{eqS3} - \eqref{eqS5} are $O\left( \frac{X}{\log^2 X} \right)$.

\noindent Since $\frac{1}{\log^2 x} \ll \left| \frac{x f_1^{(j+1)}(x)}{f_1^{(j)}(x)} \right| \ll 1$, from \eqref{e4}, we have 
$$ \frac{|f_1'(x)|}{x^j \log^{2j} x} \ll |f_1^{(j+1)}(x)| \ll \frac{|f_1'(x)|}{x^j}.$$
From \eqref{e1}, for $x \in (X_3, X_4] \subset (X_3, 2X_3]$,
$$ \frac{|f_1'(X_3)|}{X_3^j \log^{2j +1} X_3} \ll |f_1^{(j+1)}(x)| \ll \frac{|f_1'(X_3)|}{X_3^j}.$$

\noindent To evaluate $S_3$, we need to estimate $\lambda_j:= \left| \frac{\partial^j}{\partial m^j} (qf(dm))\right| = qd^j |f_1^{(j)}(dm)|$ for $j=t + 1 \geq 2$:
$$ \frac{qd^j |f_1'(X_3)|}{X_3^{j-1} \log^{2j-1} X_3} \ll \lambda_j \ll \frac{qd^j |f_1'(X_3)| }{X_3^{j-1}}.$$
We will apply Lemma \ref{lem2.3.1} with 
$$ k=j-1 (=t), \quad \lambda = \frac{qd^{k+1} |f_1'(X_3)|}{X_3^{k} \log^{2k+1} X_3} \quad \text{ and} \quad \alpha \lambda = \frac{qd^{k+1} |f_1'(X_3)| }{X_3^{k}}.$$ 
Note that $\alpha \lambda \ll X_3^{-\epsilon}$, $\lambda \left(\frac{X_3}{d}\right)^{k+1} \gg \log^{c-(2k+1)} X_3$ and $\frac{\alpha \log^k X_3/d}{X_3/d} \ll X_3^{-\epsilon}$. 
Thus,
\begin{align*}
|S_3| &\ll \sum_{d \leq X_3^{\epsilon}}  \log X \left[ \frac{X_3}{d} \log^{(-c+3k+1)/K} X_3\right] \\
         &\ll X_3 \log^2 X \log^{(-c+3k+2)/K} X_3 \ll \frac{X}{\log^2 X}.
\end{align*}

\noindent We can evaluate $S_4$ similarly: 
for $j=t+1$, denote $\mu_j := \left| \frac{\partial^j}{\partial r^j} (qf(mr)) \right| = qm^j |f_1^{(j)} (mr)|$. 
Then
$$ \frac{qm^j |f_1'(X_3)|}{X_3^{j-1} \log^{2j-1} X_3} \ll \mu_j \ll \frac{qm^j |f_1'(X_3)| }{X_3^{j-1}}.$$
Applying Lemma \ref{lem2.3.1} with 
$$k=j-1 (= t), \quad \lambda = \frac{qm^{k+1} |f_1'(X_3)|}{X_3^{k} \log^{2k+1} X_3} \quad \text{and} \quad \alpha \lambda = \frac{qm^{k+1} |f_1'(X_3)| }{X_3^{k}},$$ 
we have 
$$|S_4| \ll \sum_{m \leq uv} |a(m)| \frac{X}{m} \log^{(-c + 3k +1)/K} X_3 \ll X \log^3 X \log^{(-c + 3k +1)/K} X_3 \ll \frac{X}{\log^2 X}.$$

\noindent To evaluate the sum $S_5$, we need to estimate $S_6$ as in the proof of Theorem \ref{thm1}:
to evaluate the sum in \eqref{eqn10} for $M \geq \sqrt{X_4}$ and the sum in \eqref{eqn14} for $M < \sqrt{X_4}$.

\noindent For $M \geq \sqrt{X_4}$, take $H= \frac{X_3}{M_1}$. Then from \eqref{eqn10}, $B \ll X \log X$.
To evaluate $A$, for $j=t+1$, denote
$$ \lambda_{1,j} := \left| \frac{\partial^j}{\partial m^j} (qf(mn+mh) - qf(mn)) \right|.$$
Then, from \eqref{est lambda1j} in the proof of Theorem \ref{thm2},
$$ \frac{qh}{M^{j-1}} \frac{|f_1'(X_3)|}{\log^{2j+1} X_3} \ll \lambda_{1,j} \ll \frac{qh}{M^{j-1}} |f_1'(X_3)|. $$
We apply now Lemma \ref{lem2.3.1} with 
$$k = j -1 (= t), \, \lambda = \frac{qh}{M^{k}} \frac{|f_1'(X_3)|}{\log^{2k+3} X_3} \,\, \text{ and} \,\, \alpha \lambda = \frac{qh}{M^{k}} |f_1'(X_3)|.$$ 
Note that 
$$\alpha \lambda \ll X_3^{-\epsilon}, \quad \frac{\alpha \log^k M}{M} \ll X_3^{-\epsilon} $$
$$ \text{and} \,\, (\lambda M^{k+1})^{-1/K} (\log M)^{k/K} \ll \left( \frac{X_3}{h M} \right)^{1/K} \log^{-9} X_3.$$
Hence
\begin{align*}
A &\ll \frac{X_4}{MH}\sum_{h=1}^H \left( \frac{1}{h} \right)^{1/K} \sum_{\frac{X_3}{M_1} < n, n+h \leq \frac{X_4}{M}} \left(M \left(\frac{X_3}{M}\right)^{1/K} \log^{-9} X_3 \right)\\
   &\ll \frac{X_4}{MH} H^{1- 1/K} \frac{X_3}{M_1} \log X_3 \left(M \left(\frac{X_3}{M}\right)^{1/K} \log^{-9} X_3 \right) \ll \frac{X^2}{M} \log^{-8} X,
\end{align*}
thus in this case
$$|S_6|^2 \ll M \log^3 u |A+B| \ll M \log^3 X |A+B| \ll X^2 \log^{-8} X.$$

\noindent Similarly, for $M < \sqrt{X_4}$, take $H = M/2 = \frac{X_3}{N_1}$. 
Then from \eqref{eqn14} $D \ll X \log^3 X.$
To evaluate $C$, for $j=t+1$, denote
$$ \mu_{1,j} := \left| \frac{\partial^j}{\partial n^j} (qf(mn+mh) - qf(mn)) \right|.$$
Then, from \eqref{est mu1j} in the proof of Theorem \ref{thm2},
$$ \frac{qh}{N^{j+1}} X_3 \frac{|f_1'(X_3)|}{\log^{2j+1} X_3} \ll \mu_{1,j} \ll \frac{qh}{N^{j+1}} X_3 |f_1'(X_3)|.$$
Applying again Lemma \ref{lem2.3.1} with $k = j -1 (= t)$, we get 
$$ \lambda= \frac{qh}{N^{k+2}} X_3 \frac{|f_1'(X_3)|}{\log^{2k+3} X_3} \,\, \text{and} \,\,  \alpha \lambda= \frac{qh}{N^{k+2}} X_3 |f_1'(X_3)|  \quad \text{for} \, \, C.$$
Note that
$$\alpha \lambda \ll X_3^{-\epsilon}, \quad \frac{\alpha \log^k N}{N} \ll X_3^{-\epsilon}$$
 $$\text{and} \,\,\, (\lambda N^{k+1})^{-1/K} (\log N)^{k/K} \ll \left(\frac{N}{h} \right)^{1/K} \log^{-9} X_3,$$
thus 
$$C \ll \frac{X_4}{NH} H^{1-1/K} \frac{X_3}{N_1} \log^3 X_3 \left(N^{1+1/K} \log^{-9} X_3 \right)$$
so in this case,
$$|S_6|^2 \ll N \log N (C+D) \ll X^2 \log^{-5} X.$$
Hence we have $|S_5| \ll X \log^2 X$. 
So we are done with Case I.

\noindent Case II. Assume that $|f_1'(x)| \ll \frac{\log^c x}{x}$ eventually for some $c >0$. 

%$$ Q \leq X |f_1'(X)| \quad \text{and} \quad
%Q \leq \left| \frac{X}{\log X} \, f_1' \left( \frac{X}{\log X} \right) \right|^{1/3T^2} \,\, \text{where} \,\, T=2^t \, (t = \deg f_2).$$

 If all coefficients of $f_2(x)$ are rational, write $f_2(x) = \frac{1}{r} \sum_{j=1}^t a_j x^j$, where $a_1, \dots, a_k, r$ are integers. 
From Lemma \ref{lemmaS}, we will prove that 
\begin{equation}
\label{eqnS}
\left|\sum_{X_0 < p \leq X} e(qf_1(p) + qf_2(p)) \right| \ll \frac{\pi(X)}{Q},
\end{equation}
where $Q=Q(X) \leq \log X$ is a positive unbounded increasing function with $Q(X) \leq X |f_1'(X)|$. 
Note that
$$\left|\sum_{X_0 < p \leq X} e(qf_1(p) + qf_2(p)) \right| \ll \sum_{\substack{a=1 \\ (a,r)=1}}^r \left|\sum_{\substack{X_0 < p \leq X \\ p \equiv a \bmod \, r}} e(qf_1(p)) \right|.$$
Write $$S(a) := \left|\sum_{\substack{X_0 < p \leq X \\ p \equiv a \bmod \, r}} e(qf_1(p)) \right| = \sum_{X_0 < n \leq X} \Lambda_1(n,r,a) e(qf_1(n)).$$

\noindent Applying Lemma \ref{SW} and arguing as in the proof of the estimate for $S_0$ in $\eqref{S_0}$ above we obtain
$$|S(a)| \ll \frac{1}{\log X} \frac{1}{q |f_1'(X)|} + \frac{\pi(X)}{Q} \ll \frac{\pi(X)}{Q}.$$
So this takes care of the case when all coefficient of $f_2(x)$ are rational.  

 Now it remains to consider the case when at least one coefficient of $f_2(x)$ is irrational. We will show that for any given non-zero integer $q$,
\begin{equation}
\label{eq:limit}
 \lim_{X \rightarrow \infty} \frac{1}{\pi(X)} \left| \sum_{X_0 < p \leq X} e(qf_1(p) + qf_2(p))\right| = 0,
\end{equation}
which obviously implies
\begin{equation*}
 \lim_{X \rightarrow \infty} \frac{1}{\pi(X)} \left| \sum_{p \leq X} e(qf_1(p) + qf_2(p))\right| = 0.
\end{equation*}

\noindent For any positive integer $Y$, write all coefficients of $f_2(x)$ in the form 
\begin{equation}
\label{eq2.3.1}
\begin{split}
\alpha_j = \frac{e_j}{q_j} + z_j \quad \text{with} \quad (e_j, q_j) =1, \quad 0 < q_j \leq \tau_j := Y^{j/2} \\
|z_j| \leq \frac{1}{q_j \tau_j}, \quad z_j = \delta_j Y^{-j} \quad (j=1,2, \dots, t).
\end{split}
\end{equation}
Denote 
\begin{equation}
\label{deltaQ}
\delta_0 = \max \{ |\delta_1|, \dots , |\delta_t| \} \quad \text{ and} \quad \mathcal{Q}=[ q_1,\dots , q_t ].
\end{equation}
Assume first that there exists some $Y$ with $X_0 \leq Y \leq X$, $(X_0 = \frac{X}{\log X})$, such that 
\begin{equation}
\label{eq3.26} 
\delta_0(Y) \leq \exp \left( ( \log X)^{1/3} \right) \quad \text{ and} \quad \mathcal{Q}(Y) \leq \exp \left( (\log \log X)^3 \right).
\end{equation}
To simplify our notation, we use $\delta_0, \mathcal{Q}$ instead of $\delta_0(Y)$ and $\mathcal{Q}(Y)$.
Also, denote by $\sideset{}{^*}\sum$ the sum over all $a$ from the reduced residue system modulo $\mathcal{Q}$. We obtain
\begin{equation*}
\label{eq3.27} 
S = \sum_{X_0 < p \leq X}  e(q f_1(p) + q f_2(p)) = \sideset{}{^*}\sum \sum_{X_0 <  n \leq X} \Lambda _1 (n,\mathcal{Q},a) e(q f_1(n) + q f_2(n)).
\end{equation*}
Denoting $h(n) = \sum_{j=1}^t  z_ j n^j$, we get 
%for $n \equiv a \bmod \, \mathcal{Q}$, we get 
%$$f_2(n) - h(n)  = \sum_{j=1}^t \frac{e_ j}{ q _ j} a^j \,\,  (\bmod 1).$$
 %Put this into \eqref{eq3.27} and obtain
\begin{equation}
\label{eq3.28}
 S = \sideset{}{^*}\sum_a  e \left( \sum_{j=1}^k \frac{ e_ j}{ q _ j} a^j \right) \sum_{X_0 <n \leq X} \Lambda_1(n,\mathcal{Q},a) e(qf_1(n)+ qh(n)).
\end{equation} 

\noindent Now we will evaluate $\sum_{n=X_0}^{X_1} \Lambda_1(n,\mathcal{Q},a) e(qf_1(n)+ qh(n))$  
using the second part of Lemma \ref{lem5} with $b_n = \Lambda_1 (n,\mathcal{Q},a)$ 
and $a_n =e(q f_1 (n) + q h( n ) )$. Lemma \ref{lem2} implies that 
$$ \sum_{n \leq x} b_n = \frac{1}{\phi (\mathcal{Q})}  Li(x) - E \, \chi(a) \int_2^x  \frac{u^{\beta_1 -1}}{\log u} \, d u  + R(x) $$
with $|R(x)| \ll X \exp( - C ( \log X)^{1/2})=R$. Also, $|a_n|=1$ and 
\begin{align*}
\sum_{n \leq X} |a_n - a_{n+1}| &= \sum_n | e(q f_1(n)- q f_1(n+1) + q h(n) - q h(n+1) ) - 1 | \\
                                   &\ll \sum_n  q | f_1(n)-f_1(n+1)| + q |h(n) - h(n+1) | \\
                                   &\ll q (\log X)^{c+1}  + q \sum  |z_ j | X^{j} \ll q (\log X)^{c+1} + q \sum |z_j| Y^j \log^j X\\
                                   &\ll q ((\log X)^{c+1} + \delta_0 \log^t X)  \ll  \exp( (\log X)^{\frac{1}{3} + \epsilon_1})
\end{align*}
for some small $\epsilon_1 >0$.
Using this and Lemma \ref{lem5} we obtain
\begin{equation*}
\begin{split} \sum a_n b_n  = &\frac{1}{\phi(\mathcal{Q})} \sum_n \left( Li(n) - Li (n-1) - E \chi(a) \int_{n-1}^n \frac{u^{\beta_1-1}}{\log u}\, du \right) e(qf_1(n)+qh(n)) \\
                                                &+ O \left( X \exp \left( ( \log X)^{\frac{1}{3} + \epsilon_1} - C( \log X)^{1/2} \right) \right). 
\end{split}
\end{equation*}
Here $Li(n) - Li(n-1) = \frac{1}{\log n}+ O \left(\frac{1}{ n \log n}\right)$ and 
$\int_{n-1}^n \frac{u^{\beta_1-1}}{\log u} \, du = \frac{n^{\beta_1-1} } {\log n}  + O \left( \frac{1}{n \log n} \right).$
So, we  get  
\begin{align}
\label{above}
&\sum  a_ n  b_ n  = \frac{1}{\phi (\mathcal{Q})}  \sum_n \frac{1}{ \log n} ( 1 - E \chi(a) n^{\beta_1 - 1}) \, e(q f_1(n) + q h(n) )  + O \left( X \exp \left( - \frac{C}{2} ( \log X)^{1/2} \right) \right) \notag \\
&= \frac{1}{\phi (\mathcal{Q})}  \sum_n \frac{1}{ \log n} e(q f_1(n) + q h(n) ) - \frac{E \chi(a)}{\phi (\mathcal{Q})}  \sum_n \frac{n^{\beta_1 - 1}}{ \log n} e(q f_1(n) + q h(n) )  \\
  &+ O \left( X \exp \left( - \frac{C}{2} ( \log X)^{1/2} \right) \right). \notag
\end{align}

\noindent By using the definition of $h(n)$ above, we can write 
$$f_2(n) - h(n)  = \sum_{j=1}^t \frac{e_ j}{ q _ j} a^j \,  (\bmod \, 1) \quad \text{ for} \quad n   = a \, (\bmod \, \mathcal{Q}).$$
Using summation by parts on each of the sums in \eqref{above} above and putting the obtained expression into \eqref{eq3.28}, we obtain
\begin{align}
\label{eq3.34}
 S &\ll \frac{1}{\log X} \max_{ X_2 \in (X_0, X ]}  \left| \sum_{n \in (X_0 , X_2]}  e(q f_1(n) + q f_2(n) ) \right| +  \frac{X}{ \log^2 X} \notag \\
    &\ll \frac{1}{\log X} \left( \frac{X}{\log X} + \max_{X_2 \in (X_0, X ]} \left| \sum_{n \leq X_2}  e(q f_1(n) + q f_2(n) ) \right| \right) + \frac{X}{ \log^2 X}.
\end{align}
\noindent  By Boshernitzan's theorem, 
 $$ \lim_{Y \rightarrow \infty} \frac{1}{Y} \left| \sum_{n \leq Y} e(qf_1(n) + qf_2(n))\right| = 0,$$
therefore, from \eqref{eq3.34} 
$$\lim_{X \rightarrow \infty} \frac{1}{\pi(X)} \left| \sum_{X_0 < p \leq X} e ( q f_1(p) + q f_2(p) ) \right| = \lim_{X \rightarrow \infty} \left| \frac{S}{\pi(X)} \right| =0.$$

%Note that $\frac{d^{t+1}}{dn^{t+1}} ( q f_1(n) + q f_2(n) )  = q f_1^{(t+1)} (n)$. Since  $f_1' (x ) \ll \frac{\log^c x}{x}$ implies that
%$f_1(x) \ll \log^{c+1} x$, we get  
%$$ 0 = \lim_{x \rightarrow \infty} \frac{ \log f_1(x)}{\log x} = \lim_{x \rightarrow \infty} \frac{x f_1'(x)}{f_1 (x)} = 1+ \lim_{x \rightarrow \infty} \frac{x f_1''(x)}{f_1'(x)} = \cdots = %t+\lim_{x \rightarrow \infty} \frac{x f_1^{(t+1)} (x)}{ f_1^{(t)} (x)}.$$
%So  $1 \ll \left| \frac{x^t f_1^{(t+1)} (x)}{f_1' (x)} \right| \ll 1$ and therefore
%\begin{equation*}
%| x^{t+1} f_1^{(t+1)} (x)  | \geq | x f_1'(x) | \geq | X_0 f_1' (X_0) | \geq Q^{3 T^2} \quad (T= 2^t).
%\end{equation*}
%Now partition the interval $(X_0 , X_2]$ in \eqref{eq3.34} into subintervals of the form $(X_3,  X_4] \subset (X_3, 2 X_3]$ and evaluate the corresponding sum $S_2 := \sum_{X_3 < n \leq %X_4} e(q f_1(n) + q f_2(n))$ with the help of inequality \eqref{(1)} (see Lemma \ref{lem2.3.1} above) where we put $k=t$, $\lambda = \frac{| f_1'(X_4) |}{X_4^t}$, and $\alpha \lambda %= \frac{| f_1'(X_3) |}{X_3^t}$:
%\begin{equation}
%|S_2| \ll X_3 \left[ \frac{1}{X_3^T} + \frac{1}{Q^{3T/2}} + \frac{\log^2 X_3}{X_3^T}\right]. 
%\end{equation}
%From \eqref{eq3.34} we obtain 
%\begin{equation} S \ll \frac{\pi (X)}{Q}. \end{equation}

Finally, assume that for each $Y$ with $ X_0 \leq Y \leq X$, ($ X_0 = \frac{X}{\log X}$), 
\begin{equation}
\label{eq3.36}
\delta_0(Y) \geq \exp ( (\log X)^{1/3} ) \quad \text{ or} \quad  \mathcal{Q}(Y) \geq \exp ( ( \log \log X)^3).
\end{equation}
For \eqref{eqnS}, we use the partial summation formula with $a_n = e(qf_1(n))$ for $n \geq X_0$ and $a_n=0$ for $n < X_0$ and $b_n = \Lambda_1(n) e(q f_2(n))$:
\begin{equation*}
 S = | \sum_p e ( q f_1(p) + q f_2(p) )| \ll  qX \max_{n \in [X_0, X]} |f_1(n+1)- f_1(n)| \max_{X_2 \leq X}  \left| \sum_{p \leq  X_2}  e(q f_2(p)) \right|.
\end{equation*}
Now we use Lemma \ref{lem2.3.3} to  evaluate the last exponential sum. 
Since $f_2(x)  = f_2(x) + x^{12} \, \bmod \, 1$, we can assume that $t \geq 12$. 
We claim that 
\begin{equation}
\label{eq3.37}
\Delta_1  \ll \exp(-( \log \log X)^3 / (3t) ).
\end{equation}
For $\mathcal{Q} \leq  X^{\nu/5}$ and $\delta_0 \leq X^{\nu}$ (recall that $\nu = 1/t$), 
if $\delta_0 \geq 1$, then $\Delta_1= \delta_0^{-\nu/2} \mathcal{Q}^{-v/2 + \epsilon_0}$ satisfies \eqref{eq3.37} 
and if $\delta_0 <1$, for $m \leq X$, $\Delta_1 = (m, \mathcal{Q})^{\nu/2} \mathcal{Q}^{-\nu/2 +\epsilon_0}$ satisfies \eqref{eq3.37} too. 
For $\mathcal{Q} \geq  X^{\nu/5}$ or $\delta_0 \geq X^{\nu}$  the claim also follows from $X^{\rho} \gg X^{ 1/(20t^3)}$. 
From \eqref{eq3.37}, $(\Delta_1)^{-2} \geq \log X$ for sufficiently large $X$, so we can estimate  $\left| \sum_{p \leq  X_2}  e(q f_2(p)) \right|$ via Lemma \ref{lem2.3.3}.
Using this and $$q \max_{n \in[X_0, X_1]} |f_1(n+1) - f_1(n)| \ll q f_1'(X_0) \ll \frac{\log^{c+2} X}{X},$$
we get  
\begin{equation*}
S \ll X (\log X)^{c+2} \exp \left( - ( \log \log X)^3 / (3t) + (\log \log X)^2 / \log (1+ \epsilon_0) \right) \ll \frac{\pi(X)}{ \log X}.
\end{equation*}
 This completes the proof.
\end{proof}

\begin{Theorem} 
\label{thm4}
 Let $f_1(x)$ be a $C^1$-function such that for some $C >0$, $|x f_1 ' (x)| \leq C$ for all $x$ 
and let $f_2 (x) = \alpha_1 x + \cdots + \alpha_k x^k$ be a polynomial with at least one irrational coefficient. 
Put $f(x)=f_1(x)+ f_2(x)$. Then $(f(p))_{p \in \mathcal{P}}$ is u.d. $\bmod \, 1$.
\end{Theorem}

\begin{proof} 
By Lemma \ref{lemmaS} it is enough to prove that there exists some $Q := Q(X)$ for large $X$ such that $Q(X) \leq \log X$, $Q(X) \uparrow \infty$ and for any $q \leq Q$,
\begin{equation}
\label{eq3.45}
S := \left| \sum_{ X_0 < p \leq  X}  e(q f ( p ) ) \right| \ll \frac{\pi(X)}{Q},
\end{equation}
where $X_0 = \frac{X}{Q}$.

The estimation of the last sum is very similar to the estimation in Theorem \ref{thm3}. 
As in \eqref{eq2.3.1}, for any $\frac{X}{\log^4 X} \leq Y \leq X$, we obtain $\delta_0 = \delta_0(Y)$ and $\mathcal{Q} = \mathcal{Q}(Y)$ as following:
\begin{equation}
\label{eq3.41}
\begin{split}
&\alpha_j = \frac{e_j}{q_j} + z_j \quad \text{with} \quad (e_j, q_j) =1, \quad 0 < q_j \leq \tau_j := Y^{j/2} \\
&|z_j| \leq \frac{1}{q_j \tau_j}, \quad z_j = \delta_j Y^{-j} \quad (j=1,2, \dots, k)\\
&\delta_0 = \max \{ \delta_1, \dots , \delta_k \} \quad \text{ and} \quad \mathcal{Q}=[ q_1,\dots , q_k ]
\end{split}
\end{equation}

We will treat separately the following three cases:
\begin{enumerate}[(1)]
\item $\delta_0  \mathcal{Q} \geq  \exp( (\log \log X)^3)$ for any $\frac{X}{\log^4 X} \leq Y \leq X$
\item $\mathcal{Q} \geq \exp ((\log \log X)^3)$  for any $\frac{X}{\log^4 X} \leq Y \leq X$
\item $\delta_0 \mathcal{Q} \leq \exp( ( \log \log X)^3)$ and $\mathcal{Q} \leq \exp ((\log \log X)^3)$ for some $\frac{X}{\log^4 X} \leq Y \leq X$
\end{enumerate}

Let us first consider cases (1) and (2).
Using Lemma \ref{lem5} with $a_n = e(qf(n))$ for $n \geq X_0$ and $a_n=0$ for $n < X_0$ and $b_n = \Lambda_1 (n) e(q f_2(n))$, 
for any $Z$, we obtain 
$$\left| \sum_{Z \leq n \leq 2Z} e(qf(p)) \right| \ll q \max_{X_1 \leq 2Z} \left| \sum_{n \leq X_1} \Lambda_1(n) e(qf_2(n)) \right|,$$
since $|a_n| =1$ and $\sum\limits_{Z \leq n \leq 2Z} |a_{n+1} - a_n| \ll q $.
Thus, from \eqref{eq3.45}, dividing the interval $(X_0, X]$ into $\log Q$ subintervals of the form $(Z,2Z]$, we obtain
\begin{equation*}
%\label{3.41}
S \ll q \, \log Q  \max_{X_1 \leq X} \left| \sum_{ n \leq X_1} \Lambda_1 (n) e(q f_2(n)) \right|.
\end{equation*}

\noindent If $X_1 \leq \frac{X}{\log^4 X}$, then we have 
\begin{equation}
\label{small}q \log Q \, \left| \sum_{ n \leq X_1} \Lambda_1 (n) e(q f_2(n)) \right| \leq q \, \log Q \, \pi(X_1) \ll \frac{X}{\log^3 X}. 
\end{equation} 

Assume now that $X_1 \geq \frac{X}{\log^4 X}$.
If $\mathcal{Q}(X_1) > X^{1/5k}$  or $\delta_0(X_1) > X^{1/k}$ or if $\mathcal{Q}(X_1) \leq X^{1/5k}$ and $\delta_0(X_1) \leq X^{1/k}$ but (1) or (2) are satisfied 
then the number $\Delta_1$ from Lemma \ref{lem2.3.3} is $\ll \exp(( \log \log X)^3 ( \epsilon - 1/2k ))$ for some small $\epsilon > 0$. 
Then 
$$q \log X \left| \sum_{ n \leq X_1} \Lambda_1 (n) e(q f_2(n)) \right| \ll \frac{X}{\log^3 X}.$$
This gives $S \ll \frac{X}{\log^3 X}$.

We turn now our attention to the case (3). 
Note that $\mathcal{Q} (Y) \leq \exp(\log \log X)^3$ for some $Y$. 
As we did in Theorem \ref{thm3} for the case in \eqref{eq3.26}, we get the same inequality as in \eqref{eq3.34}:
\begin{equation}
\label{eq0915}
  S   \ll \frac{1}{\log X} \max_{X_1 \in (X_0, X]}  \left| \sum_{X_0 < n \leq  X_1}  e(q f(n)) \right| + \frac{X}{\log^2 X}.
\end{equation}
We can assume that $\alpha_k$, the leading coefficient of $f_2(n)$, is irrational, (otherwise, we can divide the interval $(X_0,X]$ in the sum \eqref{eq0915} into residue classes and remove
$\alpha_k n^k$.) We will show \eqref{eq3.45} with the additional restriction $Q (X) \leq  \left(q_k (\frac{X}{\log^4 X}) \right)^{1/4K}$. Note that $Q(X) \leq (q_k(Y))^{1/4K}$. 
For convenience of notation, let $q_k := q_k(Y)$.

We denote the sum $\left| \sum_{X_0 \leq n \leq  X_1}  e(q f(n)) \right|$ in \eqref{eq0915} by $S_1$.
Let us first consider $k > 1$. 
Using Lemma \ref{lem6} with $H_1 =( q_k)^{1/2K}$ we obtain 
\begin{equation}
\label{3.42}
\left| \frac{S_1}{ X_1} \right|^2 \ll  \frac{1}{ H_1} + \frac{1}{ X_1 H_1}  \sum_{ h_1 =1}^{ H_1} \left| \sum_{n} e(qf(n+h_1) - q f(n)) \right|.
\end{equation}
Here  $| qf_1(n+h_1) - q f_1(n) | \ll q h_1 / X_0$ so if we  remove $q f_1 ( n+h_1) - q f_1(n)$  from the last exponential sum, 
we will make an error $\ll  q H_1 / X_0 \ll 1/ H_1 $, which allows us to replace $f$ with $f_2$ in \eqref{3.42}.

Applying Lemma \ref{lem6} $(k-2)$ more times with $H_j =( H_1)^{J/2} \, ( j=1,\dots, k-1)$, where $J = 2^j$ and $K = 2^k$ we obtain 
\begin{equation}
\left| \frac{S_1}{ X_1} \right|^{K/2} \ll 
\frac{1}{H_1^{K/4}} + \cdots + \frac{1}{H_{k-1}} + \frac{1}{  X_1 H_1 \cdots H_{k-1}} \sum_{h_1 = 1}^{H_1} \cdots \sum_{h_{k-1} =1}^{H_{k-1}}  \left| \sum_n  e( q p_ 1 (n)) \right|,
\end{equation}
where $p_1 (x) = h_1 ... h_{k-1} \int_0^1 \cdots \int_0^1  \frac{\partial^{k-1}}{\partial x^{k-1}} f_2( x + h_1 t_1 + \cdots + h_{k-1} t_{k-1})  d t_1 \cdots d t_{k-1}$.
Here $\| \frac{d p_1(x)}{dx} \|= \| \alpha_k  k!  h_1 \cdots h_k \| \gg 1/ q_k$, 
so if $q_k$ is large enough that $q \leq q_k^{1/4}$, then $\| q \frac{d p_1 (x)}{ dx} \| \gg  \frac{1}{q_k}$. 
Thus by Lemma \ref{lem3} we obtain $| \sum_n  e(q p_1(n)) | \ll q_k$, which implies  
$$\left| \frac{S_1}{X_1} \right|^{K/2} \ll  q_k^{-1/8} + \frac{q_k}{X_1} \leq q_k^{-1/8}$$ 
since $q_k \leq \mathcal{Q} \leq \exp((\log \log X)^3)$. Thus $S_1 \leq X_1 q_k^{-1/4K} \ll X/Q$, so we have $S \ll \frac{\pi(X)}{Q}$.  

Now consider  $k =1$. 
We divide the interval $(X_0, X_1]$ in \eqref{eq0915} into $\log Q$ subintervals of the form $(Z, 2Z]$ 
and apply Lemma \ref{lem5} with $a_n = e(qf_1(n))$ and $b_n = e(qf_2(n))$ to obtain
$$\left| \sum_{X_0 < n \leq X_1} e(q f(n)) \right| \ll q \log Q \max_{X_2 \leq X_1} \left| \sum_{X_0 \leq n \leq X_2} e(qf_2(n)) \right|.$$
From $\| f_2'(x) \| \gg \frac{1}{q_1}$, we have $|\sum e(qf_2(n))| \leq q_1$. 
Thus, 
$$S \ll q \frac{ \log Q}{\log X} q_1 + \frac{X}{\log^2 X} \ll q \frac{\log Q}{\log X} \exp ((\log \log X)^3) + \frac{X}{\log^2 X} \ll \frac{X}{\log^2 X}.$$
This concludes the proof of Theorem \ref{thm4}.
\end{proof}

\begin{Remark} Note that Theorem \ref{thm4} implies (and provides a new rather short proof of) the classical result of Rhin \cite{Rh}, which states 
that if $f(x)$ is a polynomial with at least one coefficient other than the constant term irrational, then $(f(p))_{p \in \mathcal{P}}$ is uniformly distributed $\bmod \, 1$. If one takes Rhin's theorem for granted, then the proof of Theorem \ref{thm4} can be made shorter. Indeed, using 
summation by parts with $a_n = e(q f_1(n))$ and $b_n = \Lambda(n) e(qf_2(n))$ we get
\begin{equation*}
\begin{split}
 &\sum_{n=1}^N \Lambda (n) e( q f_1(n) + qf_2(n)) = \\
 &\sum_{n=1}^{N-1} [e(qf_1(n)) - e (qf_1(n+1))] \sum_{m=1}^n \Lambda(m) e(qf_2(m)) + e(q f_1(N)) \sum_{m=1}^N \Lambda(m) e(q f_2(m)).
\end{split}
\end{equation*}
Note that $$|e (qf_1(n)) - e (qf_1(n+1))| \leq 2 \pi q |f_1(n) - f_1(n+1)| \leq \frac{2 \pi q C}{n}.$$
Denote $c_n =(e(qf_1(n)) - e (qf_1(n+1))) \sum_{m=1}^n \Lambda(m) e(qf_2(m))$. Then, by Rhin's theorem,
$$|c_n| \ll \left| \frac{1}{n}  \sum_{m=1}^n \Lambda(m) e(qf_2(m)) \right| \rightarrow 0.$$ 
Therefore, for any non-zero integer $q$,
$$ \lim_{N \rightarrow \infty } \frac{1}{N} \sum_{n=1}^N \Lambda (n) e( q f(n)) =0,$$
so we are done.
\end{Remark}

We are now in the position to prove the main theorem in this section:
\begin{proof}[Proof of Theorem \ref{main}] We need to prove that (2) and (3) are equivalent. 
Let us first show that $(3) \Rightarrow (2)$:
we can write $f(x) = f_1(x) + f_2(x)$, where
\begin{itemize}
\item $f_1(x)$ is of the type $x^{l+}$ for some non-negative integer $l$ or $\lim\limits_{x \rightarrow \infty} f_1(x) = c$ for some $c \in \mathbb{R}$,
\item $f_2(x) \in \mathbb{R}[x]$. %Moreover, if $\lim_{x \rightarrow \infty} f_1(x) =c < \infty$, then $f_2(x) -f_2(0)$ has at least one irrational coefficients.
\end{itemize}

Case I. $l \geq 1$: the result follows from Theorems \ref{thm2} and \ref{thm3}.

Case II. $l =0$:  
If $\lim_{x \rightarrow \infty } \left| \frac{f_1(x)}{\log x} \right| = \infty$, 
the result follows from Theorems \ref{thm1} and \ref{thm3}. 
Otherwise, $\lim_{x \rightarrow \infty} |x f_1'(x)| = \lim_{x \rightarrow \infty } \left| \frac{f_1(x)}{\log x} \right| \leq C$ for some $C$. 
Then $f_2(x) - f_2(0)$ has at least one irrational coefficient, so the result follows from Theorem \ref{thm4}.

Case III. If $\lim_{x \rightarrow \infty} f_1(x) = c$ for some $c \in \mathbb{R}$, then, again, the result follows from Theorem \ref{thm4}.

Now let us show that $(2) \Rightarrow (3)$:  
Suppose not. 
There exists a function $f(x)$ such that $(f(p))_{p \in \mathcal{P}}$ is u.d. $\bmod \, 1$ 
and $\lim\limits_{x \rightarrow \infty} \frac{f(x) - P(x)}{\log x} = A (\ne \pm \infty)$ for some $P(x) \in \mathbb{Q}[x]$. 
Find an integer $q$ such that $qP(x) \in \mathbb{Z}[x]$ and let $g(x) = q f(x) - q P(x)$. 
Then $(g(p))_{p \in \mathcal{P}}$ is u.d. $\bmod \, 1$ since $g(p) \equiv qf(p) \, \bmod \, 1$, thus for any non-zero integer $h$,
$$\lim_{N \rightarrow \infty} \frac{1}{\pi(N)} \sum_{p \leq N} e(h g(p)) = 0.$$
We have $\lim\limits_{x \rightarrow \infty} \frac{g(x)}{\log x} \ne \pm \infty$, so $|g'(x)| \ll \frac{1}{x}$.
%Use the same argument to obtain \eqref{Con} as in the proof of Theorem 3.3, for any non-zero integer $h$, we have 
The argument which was utilized in the course of the proof of Theorem \ref{thm1} to establish the formula \eqref{Con} gives us the following estimates for any non-zero integer $h$
\begin{align*}
 \sum_{p \leq N} e(h g(p)) 
    &= \sum_{N_0 < n \leq N} \Lambda_1(n) e(h g(n)) + O \left( \frac{\pi(N)}{\log N} \right) \quad \quad \left( N_0 = \frac{N}{\log N} \right) \\
    &= \sum_{N_0 < n \leq N} \frac{1}{\log n} e (h g(n)) + O \left( \frac{\pi(N)}{\log N} \right) \\
    &= \sum_{ n \leq N} \frac{1}{\log n} e (h g(n)) + O \left( \frac{\pi(N)}{\log N} \right).
\end{align*}
Then $$\lim_{N \rightarrow \infty} \frac{\log N}{N} \sum_{n \leq N} \frac{e(h g(n))}{\log n} = 0.$$
Using summation by parts, we get
$$\left|\sum_{n \leq N} e(h g(n)) \right| = \left| \sum_{n \leq N} \frac{e(h g(n))}{\log n} \log n \right| \ll \log N \max_{N_0 \leq N} \left| \sum_{n \leq N_0} \frac{e (h g(n))}{\log n}\right|.$$
Thus $$\lim_{N \rightarrow \infty} \frac{1}{N} \sum_{n \leq N} e (h g(n)) = 0,$$
which implies that $(g(n))_{n \in \mathbb{N}}$ is u.d. $\bmod \,1$. 
By the result of Boshernitzan \cite{Bos} alluded to in the introduction, this implies $\lim\limits_{x \rightarrow \infty} \frac{f(x) - P(x)}{\log x} = \pm \infty$, which contradicts the above assumption that $A \ne \pm \infty$. We are done.
%But it contradicts Theorem \ref{Bos}.
\end{proof}

%%%%%%%%%%%%%%%%%%%%%%%%%%%%%%
\section{Applications}
\label{sec : app}
For a given Hardy field $H$, let ${\bf H}$ be the set of all subpolynomial functions $\xi(x) \in H$ such that 
$$ \text{either} \,\,\, \lim\limits_{x \rightarrow \infty} \frac{\xi(x)}{x^{l+1}} = \lim\limits_{x \rightarrow \infty} \frac{x^l}{\xi(x)} = 0 \,\, \text{ for some} \,\, l \in \mathbb{N}, 
\,\,\, \text{or} \,\,\, \lim\limits_{x \rightarrow \infty} \frac{\xi(x)}{x} = \lim\limits_{x \rightarrow \infty} \frac{\log x}{\xi(x)} =0.$$

\subsection{Some corollaries of Theorem \ref{main}}\mbox{}

In this subsection we collect some consequences of results in previous sections, which will be utilized in subsequent subsections for derivation of ergodic and combinatorial applications.

\begin{Proposition}
\label{cor4.1}
Let $P_1(x), \dots, P_m(x)$ be polynomials such that $(P_1(n), \dots, P_m(n))_{n \in \mathbb{N}}$ is u.d.\ $\bmod \, 1$ in $\mathbb{T}^m$ and let $\xi_1(x), \dots , \xi_k(x) \in {\bf H}$ such that $\sum_{i=1}^k b_i \xi_i(x) \in {\bf H}$ for any $(b_1, \dots, b_k) \in \mathbb{Z}^k \backslash \{(0, 0, \dots, 0)\}$. 
Then $( P_1(p), \dots, P_m(p), \xi_1(p), \dots, \xi_k(p))_{p \in \mathcal{P}}$ is u.d. $ \bmod \, 1$ in $\mathbb{T}^{m+ k}.$
\end{Proposition}

\begin{proof}
By Weyl's criterion, it is enough to show that 
$$( a_1 P_1(p) + \cdots + a_m P_m(p) + b_1 \xi_1(p) + \cdots + b_k \xi_k(p))_{p \in \mathcal{P}} \, \, \text{ is u.d.} \,\, \bmod \, 1,$$
where $a_i, b_j \in \mathbb{Z}$ and at least one of $a_i$ or $b_j$ is not $0$, which follows from Theorem \ref{thm3}.
\end{proof}

\begin{Remark} 
Here are some examples of $\xi_1(x), \dots, \xi_k(x) \in \bold{H}$ satisfying the condition in Proposition \ref{cor4.1}.
\begin{itemize}
\item If $\xi_1(x), \dots, \xi_k(x) \in \bold{H}$ have {\it different growth rates}, that is,  
$$\lim\limits_{x \rightarrow \infty} \frac{\xi_i(x)}{\xi_j(x)} = \pm \infty \quad \text{or} \quad 0 \quad \quad \forall i \ne j,$$
then $\sum_{i=1}^k b_i \xi_i(x) \in {\bf H}$ for any $(b_1, \dots, b_k) \in \mathbb{R}^k \backslash \{(0, 0, \dots, 0)\}$.
\item $\xi_1(x) = \sqrt{x}  + \log^2 x$ and $\xi_2(x) = 2 \sqrt{x} + \log^2 x$ do not have different growth rates, but they satisfy the condition in Proposition \ref{cor4.1}.
\end{itemize}
\end{Remark}

\begin{Proposition} 
\label{ud lemma}
Let $\xi_1(x), \dots , \xi_k(x) \in {\bf H}$ such that $\sum_{i=1}^k b_i \xi_i(x) \in {\bf H}$ for any $(b_1, \dots, b_k) \in \mathbb{R}^k \backslash \{(0, \dots, 0)\}$.
Let $g(x) = \sum_{i=1}^k \alpha_i [\xi_i(x)] $, where $\alpha_1, \dots, \alpha_k$ be real numbers and let $P(x) = a_0 + a_1 x + \cdots + a_t x^t$ be a polynomial.
Then,
\begin{enumerate}[(i)]
\item 
If $(\alpha_1, \dots , \alpha_k) \ne (0, \dots, 0)$, then
$$\lim\limits_{N \rightarrow \infty} \dfrac{1}{\pi(N)} \sum\limits_{p \leq N} e(g(p )) = 0.$$
\item If one of $a_i (1 \leq i \leq t)$ or $\alpha_i (1 \leq i \leq k)$ is irrational, then
$(g(p) + P(p))_{p \in \mathcal{P}}$ is u.d.\ $\bmod \, 1$.
\end{enumerate}
\end{Proposition} 

\begin{proof}
Without loss of generality we can assume that all $\alpha_i$ are non-zero.
Let $p_n$ be the $n$-th prime number. Theorem \ref{main} implies that
\begin{enumerate}[(a)]
\item $\big( \sum_{i=1}^k b_i \xi_i(p_n) \big)_{n \in \mathbb{N}}$ is u.d. mod 1 for any $(b_1, b_2, \dots , b_k) \ne (0, 0, \dots , 0)$.
\item $\big( \sum_{i=1}^k b_i \xi_i(p_n) + \sum_{j=0}^t c_j p_n^j \big)_{n \in \mathbb{N}}$ is u.d.\ mod 1 if $(b_1, b_2, \dots , b_k) \ne (0,0,\dots, 0)$ or one of the $c_i$ is irrational.
\end{enumerate}

Reordering $\alpha_i$, if needed, we will assume that there exists a non-negative integer $m$ such that 
$\alpha_i \notin \mathbb{Q}$ for $i \leq m$ and $\alpha_i \in \mathbb{Q}$ for $i \geq m+1$.
For rational $\alpha_i (i \geq m+1)$, one has $q \in \mathbb{N}$ such that $\alpha_i = \frac{d_i}{q}$ with $d_i \in \mathbb{Z}$.

Let us first prove $(i)$.
Define $f_i(x,y) = e(x-\{y\} \alpha_i)$ if $\alpha_i$ is irrational and define $f_i(x) = e(\alpha_i x)$  if $\alpha_i$ is rational. 
Then $e(g(p_n))$ is the product of $f_i(\alpha_i \xi_i(p_n), \xi_i(p_n)) (i \leq m)$ and $f_i ([\xi_i(p_n)]) (i \geq m+1)$.
Note that (a) implies that  for any $u \in \mathbb{N}$,
$$ \left( \alpha_1 \xi_1 (p_n), \xi_1(p_n), \alpha_2 \xi_2(p_n), \xi_2(p_n), \dots, \alpha_m \xi_m(p_n), \xi_m(p_n), \frac{\xi_{m+1}(p_n)}{u}, \frac{\xi_{m+2}(p_n)}{u}, \dots, \frac{\xi_k(p_n)}{u} \right)$$ 
is u.d. mod 1 in $\mathbb{T}^{m+k}$.
Thus,
$$ (\alpha_1 \xi_1 (p_n), \xi_1(p_n), \alpha_2 \xi_2(p_n), \xi_2(p_n), \dots, \alpha_m \xi_m(p_n), \xi_m(p_n), [\xi_{m+1}(p_n)], [\xi_{m+2}(p_n)], \dots, [\xi_k(p_n)])$$
is u.d. mod 1 in $\mathbb{T}^{2m} \times \mathbb{Z}_q^{k-m}$. 
So $(i)$ follows.

Now let us prove $(ii)$. 
If all $\alpha_i =0$, then $(ii)$ follows from Theorem \ref{thm4}. (It also follows from Rhin's theorem that $P(p_n)$ is u.d.\ mod 1 (\cite{Rh}).) 
Otherwise, without loss of generality we can assume that all $\alpha_i \ne 0$. 
We want to show that for any non-zero $r \in \mathbb{Z}$, 
\begin{equation}
\label{eq3.2}
\lim_{N \rightarrow \infty} \frac{1}{N} \sum_{n=1}^N e(rP(p_n) + rg(p_n))=0.
\end{equation}
Suppose that one of the $a_i (i \geq 1)$ is irrational. 
Let $f_0(x) = e(rx)$, $f_i(x,y) = e(r(x - \alpha_i\{y\})) (1 \leq i \leq m)$, and $f_i(x) = e(r \frac{d_i}{q}x) (i \geq m+1)$.
Then $e(rP(p_n) + r g(p_n))$ is the product of $f_0(P(p_n))$, $f_i (\alpha_i \xi_i(p_n), \xi_i(p_n)) (1 \leq i \leq m)$ and $f_i ([\xi_i(p_n)]) (i \geq m+1)$.
From $(b)$, for any $u \in \mathbb{N}$, 
$$ \left( P(p_n), \alpha_1 \xi_1 (p_n), \xi_1(p_n), \alpha_2 \xi_2(p_n), \xi_2(p_n), \dots, \alpha_m \xi_m(p_n), \xi_m(p_n), \frac{\xi_{m+1}(p_n)}{u}, \frac{\xi_{m+2}(p_n)}{u}, \dots, \frac{\xi_k(p_n)}{u} \right)$$ 
is u.d.\ mod 1 in $\mathbb{T}^{m+k+1}$. 
Thus, 
$$ (P(p_n), \alpha_1 \xi_1 (p_n), \xi_1(p_n), \alpha_2 \xi_2(p_n), \xi_2(p_n), \dots, \alpha_m \xi_m(p_n), \xi_m(p_n), [\xi_{m+1}(p_n)], [\xi_{m+2}(p_n)], \dots, [\xi_k(p_n)])$$
is u.d. mod 1 in $\mathbb{T}^{2m+1} \times \mathbb{Z}_q^{k-m}$, so \eqref{eq3.2} follows.
Finally if all $a_i$ are rational numbers, then $\alpha_1$ is irrational. 
We can show \eqref{eq3.2} similarly by using that for any $u \in \mathbb{N}$
$$ \left( P(p_n) + \alpha_1 \xi_1 (p_n), \xi_1(p_n), \alpha_2 \xi_2(p_n), \xi_2(p_n), \dots, \alpha_l \xi_m(p_n), \xi_m(p_n), \frac{\xi_{m+1}(p_n)}{u}, \frac{\xi_{m+2}(p_n)}{u}, \dots, \frac{\xi_k(p_n)}{u} \right)$$ 
is u.d.\ mod 1 in $\mathbb{T}^{m+k}$, so \eqref{eq3.2} follows.
\end{proof} 
 
\subsection{Ergodic sequences}\mbox{}  

In this subsection we deal with the sequences of the form $\bold{d}_n = ([\xi_1 (p_n)], \dots, [\xi_k(p_n)])$, where $\xi_1(x), \dots , \xi_k(x) \in {\bf H}$ and $\sum_{i=1}^k b_i \xi_i(x) \in {\bf H}$ for any $(b_1, \dots, b_k) \in \mathbb{R}^k \backslash \{(0, 0, \dots, 0)\}$.
The main result is that $(\bold{d}_n)_{n \in \mathbb{N}}$ is an {\bf ergodic sequence}:  for any ergodic measure preserving $\mathbb{Z}^k$-action $T=(T^{\bold{m}})_{(\bold{m} \in \mathbb{Z}^k)}$ on a probability space $(X, \mathcal{B}, \mu)$ and for any $f \in L^2$, 
$$\lim_{N \rightarrow \infty} \frac{1}{N} \sum_{n=1}^N f \circ T^{\bold{d}_n} = \int f \, d \mu  \,\,\, \text{in} \,\, L^2.$$ 

Recall the following version of the classical Bochner-Herglotz theorem.
\begin{Theorem}
\label{BH}
Let $U_1, \dots, U_k$ be commuting unitary operators on a Hilbert space $\mathcal{H}$ and let $f \in \mathcal{H}$. 
Then there is a measure $\nu_f$ on $\mathbb{T}^k$ such that 
$$  \langle U_1^{n_1} U_2^{n_2} \cdots U_k^{n_k} f , f \rangle  \, \, = \int_{\mathbb{T}^k} e^{2 \pi i (n_1 \gamma_1 + \cdots + n_k \gamma_k)} \, d \nu_f (\gamma_1, \dots , \gamma_k), $$
for any $(n_1, n_2, \dots , n_k) \in \mathbb{Z}^k$. 
\end{Theorem}

\begin{Theorem}
\label{ergodic}
Let $U_1, \dots, U_k$ be commuting unitary operators on a Hilbert space $\mathcal{H}$. 
Let $\xi_1(x), \dots , \xi_k(x) \in {\bf H}$ such that $\sum_{i=1}^k b_i \xi_i(x) \in {\bf H}$ for any $(b_1, \dots, b_k) \in \mathbb{R}^k \backslash \{(0, 0, \dots, 0)\}$.
Then,
\begin{equation}
\label{ergodiceq}
\lim_{N \rightarrow \infty} \frac{1}{N}  \sum_{n=1}^N U_1^{[\xi_1(p_n)]} \cdots U_k^{[\xi_k(p_n)]} f = f^*,
\end{equation}
where $f^*$ is the projection of $f$ on $\mathcal{H}_{inv} (:= \{ f \in \mathcal{H} :  U_i f = f \,\, \textrm{for all} \,\, i\})$.
\end{Theorem}

\begin{proof} We will use a Hilbert space splitting $\mathcal{H} = \mathcal{H}_{inv} \oplus \mathcal{H}_{erg}$, where
\begin{align*}
\mathcal{H}_{inv} &= \{ f \in \mathcal{H} :  U_i f = f \,\, \textrm{for all} \,\, i = 1,2, \dots, k \}, \\
\mathcal{H}_{erg} &= \{ f \in \mathcal{H} : \lim_{N_1, \dots, N_k \rightarrow \infty} \left|\left|\frac{1}{N_1 \cdots N_k} \sum_{n_1=0}^{N_1-1} \cdots  \sum_{n_k=0}^{N_k-1}U_1^{n_1} \cdots U_k^{n_k} f \right|\right|_{\mathcal{H}} = 0 \}.
\end{align*}
For $f \in \mathcal{H}_{inv}$, $U_1^{[\xi_1(p_n)]} \cdots U_k^{[\xi_k(p_n)]} f  = f$. 
So we need to prove that for $f \in \mathcal{H}_{erg}$, the left hand side in \eqref{ergodiceq} converges to $0$. 
This  follows from the Bochner-Herglotz theorem and Proposition \ref{ud lemma}(i):
\begin{align*}
&\left|\left| \frac{1}{N} \sum_{n=1}^N U_1^{[\xi_1(p_n)]} \cdots U_k^{[\xi_k(p_n)]}  f \right|\right|_{\mathcal{H}}^2 \\
&=\frac{1}{N^2} \sum_{m,n=1}^N \langle U_1^{[\xi_1(p_m)]  - [\xi_1(p_n)]}  \cdots U_k^{[\xi_k(p_m)] -  [\xi_k(p_n)]} f, f \rangle \\
&=\frac{1}{N^2} \sum_{m,n=1}^N \int e(([\xi_1(p_m)] - [\xi_1(p_n)] , \dots ,  [\xi_k(p_m)] -  [\xi_k(p_n)]) \cdot \bold{\gamma}) \, d \nu_f (\bold{\gamma}) \\
&= \int \left| \frac{1}{N} \sum_{n=1}^N e ( ([\xi_1(p_n)],  \dots , [\xi_k(p_n)]) \cdot \bold{\gamma}) \right|^2 \, d \nu_f(\bold{\gamma}) \rightarrow 0,
\end{align*}
since $f \in \mathcal{H}_{erg}$, so $\nu_f (\{0,0 \dots, 0\})=0.$
\end{proof}

\begin{Corollary}
Let $\bold{d}_n = ([\xi_1 (p_n)], \dots, [\xi_k(p_n)])$, where $\xi_1(x), \dots , \xi_k(x) \in {\bf H}$ and $\sum_{i=1}^k b_i \xi_i(x) \in {\bf H}$ for any $(b_1, \dots, b_k) \in \mathbb{R}^k \backslash \{(0, 0, \dots, 0)\}$. 
Let $T=(T^{\bold{m}})_{(\bold{m} \in \mathbb{Z}^k)}$ be an ergodic measure preserving $\mathbb{Z}^k$-action on a probability space $(X, \mathcal{B}, \mu)$. 
Then for any $f \in L^2$
$$\lim_{N \rightarrow \infty} \frac{1}{N} \sum_{n=1}^N f \circ T^{\bold{d}_n} = \int f \, d \mu  \,\,\, \text{in} \,\, L^2.$$ 
\end{Corollary}

\begin{Theorem}
\label{A}
Let $\xi_1(x), \dots , \xi_k(x) \in {\bf H}$ such that $\sum_{i=1}^k b_i \xi_i(x) \in {\bf H}$ for any $(b_1, \dots, b_k) \in \mathbb{R}^k \backslash \{(0, 0, \dots, 0)\}$. 
Let $L: \mathbb{Z}^k \rightarrow \mathbb{Z}^m$ be a linear transformation and $(\psi_1(n), \dots, \psi_m(n)) = L([\xi_1(n)], \dots, [\xi_k(n)])$. 
Let $T_1, T_2, \dots , T_m$ be commuting, invertible measure preserving transformations on a probability space $(X, \mathcal{B}, \mu)$. 
Then, for any $A \in \mathcal{B}$ with $\mu(A) > 0$, one has 
$$\lim_{N \rightarrow \infty} \frac{1}{N} \sum_{n=1}^{ N} \mu(A \cap T_1^{-\psi_1(p_{n})} \cdots T_m^{-\psi_m(p_n)} A) \geq \mu^2(A).$$
\end{Theorem}
\begin{proof}
By Theorem \ref{BH} there exists a positive measure $\nu$ on $\mathbb{T}^m$ such that 
$$\mu(A \cap T_1^{-n_1} \cdots T_m^{-n_m} A) = \int e^{2 \pi i (n_1 \gamma_1 + \cdots n_m \gamma_m)} \, d \nu (\gamma_1, \cdots, \gamma_m).$$
Also one can see that $ \nu \{(0,\dots, 0)\} \geq \mu(A)^2.$
Let $L^t: \mathbb{T}^m \rightarrow \mathbb{T}^k$ be defined by $\bold{k} \cdot L^t( \bold{x}) = L(\bold{k}) \cdot \bold{x}$, where $\bold{k} \in \mathbb{Z}^k$  and $\bold{x} \in \mathbb{T}^m$. 
Denote by $\nu'$ be the image of $\nu$ under $L^t$. 
Then $$ \hat{\nu}'(\bold{k}) = \hat{\nu}(L(\bold{k})) \quad \text{and} \quad \nu'\{ (0, \dots, 0) \} \geq \nu\{ (0, \dots, 0) \}.$$
Moreover, by Proposition \ref{ud lemma},
$$\lim_{N \rightarrow \infty} \frac{1}{N} \sum_{n=1}^N \hat{\nu}' ([\xi_1(p_n)], \dots, [\xi_k(p_n)]) = \nu'\{(0, \dots, 0)\}.$$
Thus,
\begin{align*}
&\lim_{N \rightarrow \infty} \frac{1}{N} \sum_{n=1}^{ N} \mu(A \cap T_1^{-\psi_1(p_{n})} \cdots T_m^{-\psi_m(p_n)} A) = \lim_{N \rightarrow \infty} \frac{1}{N} \sum_{n=1}^n \hat{\nu} (\psi_1(p_n), \dots, \psi_m(p_n)) \\
&= \lim_{N \rightarrow \infty} \frac{1}{N} \sum_{n=1}^N \hat{\nu}' ([\xi_1(p_n)], \dots, [\xi_k(p_n)]) = \nu'\{(0, \dots, 0)\} \geq \nu\{(0, \dots, 0)\} \geq \mu^2(A). \qedhere
\end{align*}

\end{proof}

The next result follows from Theorem \ref{A} with the help of  Furstenberg's correspondence principle. (See Proposition \ref{correspondence}.) 
\begin{Corollary}
\label{prop}
Let $\psi_1(x), \dots , \psi_k(x)$ be as in Theorem \ref{A}.
If $E \subset \mathbb{Z}^k$ with ${d^*}(E) > 0$, then there exists a prime $p$ such that $(\psi_1(p), \cdots , \psi_k(p) ) \in E - E$. 
Moreover,
$$\liminf_{N \rightarrow \infty} \frac{ | \{ p \leq N:   (\psi_1(p), \dots , \psi_k(p) ) \in E - E\} | }{ \pi(N) } \geq {d^*}(E)^2.$$
\end{Corollary}

\subsection{Nice $FC^+$ sets}\mbox{}

Before stating the main result of this subsection, we will define some relevant notions.
For $\bold{d} = (d_1, d_2, \dots, d_k) \in \mathbb{Z}^k$, we write $|\bold{d}| = \max_{1 \leq i \leq k} |d_i|$. 
Let $D$ be an infinite subset of $\mathbb{Z}^k$. 
We write $D = \{ \bold{d}_n: n \in \mathbb{N}\}$ with the convention that $\bold{d}_{n_1} \ne \bold{d}_{n_2}$ for $n_1 \ne n_2$ and $|\bold{d}_n|$ is non-decreasing.

\begin{Definition}\mbox{}

\begin{enumerate}
\item A set $D \subset \mathbb{Z}^k$ is a {\bf{set of recurrence}} 
if given any measure preserving $\mathbb{Z}^k$-action $T=(T^{\bold{m}})_{(\bold{m} \in \mathbb{Z}^k)}$ on a probability space $(X, \mathcal{B}, \mu)$ 
and any set $A \in \mathcal{B}$ with $\mu(A) > 0$, 
there exists $\bold{d} \in D$ $(\bold{d} \ne 0)$ such that 
$$\mu(A \cap T^{-\bold{d}} A) > 0. $$
%\item  A set $D \subset \mathbb{Z}^k$ is an {\bf{averaging set of recurrence}} 
%if given any measure preserving $\mathbb{Z}^k$-action $T=(T^{\bold{m}})_{(\bold{m} \in \mathbb{Z}^k)}$ on a probability space $(X, \mathcal{B}, \mu)$ 
%and any set $A \in \mathcal{B}$ with $\mu(A) > 0$, 
%we have 
%$$\limsup_{N \rightarrow \infty } \frac{1}{N} \sum_{n=1}^N \mu(A \cap T^{-\bold{d}_n} A) > 0 .$$
\item A set $D \subset \mathbb{Z}^k$ is a {\bf{set of nice recurrence}} 
if given any measure preserving $\mathbb{Z}^k$-action $T=(T^{\bold{m}})_{(\bold{m} \in \mathbb{Z}^k)}$ on a probability space $(X, \mathcal{B}, \mu)$, 
any set $A \in \mathcal{B}$ with $\mu(A) > 0$ and any $\epsilon >0$, we have 
$$\mu(A \cap T^{-\bold{d}} A) \geq \mu^2(A) - \epsilon $$
for infinitely many $\bold{d} \in D$.
\end{enumerate}
\end{Definition}

\begin{Definition}[cf. Definition 1.2.1 in \cite{BL}] 
A subset $D$ of $\mathbb{Z}^k \backslash \{0\}$ is {\bf{a van der Corput set}} ({\bf{vdC set}}) 
if for any family $(u_{\bold{n}})_{\bold{n} \in \mathbb{Z}^k}$ of complex numbers of modulus $1$ such that 
$$\forall \bold{d} \in D, \,\,  
\lim_{N_1, \dots, N_k \rightarrow \infty} \frac{1}{N_1 \cdots N_k} \sum_{\bold{n} \in \prod_{i=1}^k [0, N_i)} u_{\bold{n}+\bold{d}} \overline{u_{\bold{n}}} = 0,$$
we have $$\lim_{N_1, \dots, N_k \rightarrow \infty} \frac{1}{N_1 \cdots N_k } \sum_{\bold{n} \in \prod_{i=1}^k [0, N_i)} u_{\bold{n}} = 0.$$
\end{Definition}

\begin{Definition}
An infinite subset $D$ of $\mathbb{Z}^k$ is a {\bf {nice $FC^+$ set}} if for any positive finite measure $\sigma$ on $\mathbb{T}^k$,
$$\sigma( \{ (0,0, \dots, 0) \} ) \leq \limsup_{|\bold{d}| \rightarrow \infty, \bold{d} \in D} |\hat{\sigma}(\bold{d})|. $$
\end{Definition}
It is known that every nice $FC^+$ set is a vdC set and a set of nice recurrence (see Section 3.5 in \cite{BL} or Remark 4 in \cite{BKMST}). 

Let $P_1(x), \dots,P_l(x) \in \mathbb{Z}[x]$ with $P_i(0)=0$ for all $1 \leq i \leq l$ and let $\xi_1 (x), \dots, \xi_k(x) \in {\bf H}$  such that $\sum_{i=1}^k b_i \xi_i(x) \in {\bf H}$ for any $(b_1, \dots, b_k) \in \mathbb{Z}^k \backslash \{(0, 0, \dots, 0)\}$. Let $L$ be a non-zero linear map $\mathbb{Z}^{l+k} \rightarrow \mathbb{Z}^m.$ 
Denote 
\begin{equation}
\label{Dset}
\begin{split}
 {D}_{-1} &=  \{ \left(  p-1,  (p-1)^2, \cdots , (p-1)^l, [\xi_1(p)], \cdots , [\xi_k (p)]  \right) | \, p \in\mathcal{P} \}, \\
 {D}_{1} &=  \{ \left(  p+1, (p+1)^2, \cdots , (p+1)^l, [\xi_1(p)], \cdots , [\xi_k (p)]  \right) | \, p \in\mathcal{P} \}.
\end{split}
\end{equation}
The main result in this subsection is following. 
\begin{Theorem}
\label{sarkozy type}
 $\bold{D}_1$ and $\bold{D}_{-1}$ are nice $FC^+$ sets in $\mathbb{Z}^{m}$, and so they are vdC sets and also sets of nice recurrence.
\end{Theorem}

Let $L: \mathbb{Z}^{l+k} \rightarrow \mathbb{Z}^m$ be a non-zero linear map and let $\bold{D}_i = L(D_i)$ $(i = \pm 1)$.
%such that $\bold{D}_i \ne \{(0, \dots, 0)\}$, 
Note that $\bold{D}_i$ is an infinite set since $(x \pm i), (x \pm i)^2, \dots, (x \pm i)^l, \xi_1(x), \cdots, \xi_k(x)$ are linearly independent over $\mathbb{Z}$.  

\begin{Corollary}
Let $P_1(x), \dots, P_l(x) \in \mathbb{Z}[x]$ with $P_i(0)=0$ for all $1 \leq i \leq l$. Then the following sets are nice $FC^+$ sets:
\begin{equation*}
\label{disply:polynomial}
\begin{split}
  \{ \left( P_1( p-1), P_2 (p-1), \cdots , P_l(p-1), [\xi_1(p)], \cdots , [\xi_k (p)]  \right) | \, p \in\mathcal{P} \}, \\
 \{ \left(  P_1(p+1), P_2(p+1), \cdots , P_l(p+1), [\xi_1(p)], \cdots , [\xi_k (p)]  \right) | \, p \in\mathcal{P} \}.
\end{split}
\end{equation*}
\end{Corollary}

\noindent Note that for any $(c_1, \dots, c_l, d_1, \dots, d_k) \in \mathbb{Z}^k \backslash \{(0, 0, \dots, 0)\}$, 
$$\left| c_1 (x-i) + \cdots + c_l(x-i)^l + d_1 [\xi_1(x)] + \cdots + d_l [\xi_k(x)]\right|$$
 is eventually increasing to $\infty$, so there exists a finite set $F \subset D_i$, $(0, \dots, 0) \notin L(D_i \backslash F)$. 
Therefore the following result implies that in order to prove Theorem \ref{sarkozy type}, it is enough to show that $D_1$ and $D_{-1}$ are nice $FC^+$ sets.  
\begin{Lemma} [cf.\ \cite{BL} Corollary 1.15]
Let $D \subset \mathbb{Z}^d$ be a nice $FC^+$ set.
Let $L: \mathbb{Z}^d \rightarrow \mathbb{Z}^e$ be a linear transformation such that $(0,0,\dots, 0) \notin L(D)$. Then $L(D)$ is a nice $FC^+$ set. 
\end{Lemma}
\begin{proof}
The proof is analogous to the proof of Theorem \ref{A}.
Let $\nu$ be a positive measure on $\mathbb{T}^e$. 
Define $L^t: \mathbb{T}^e \rightarrow \mathbb{T}^d$ by $\bold{k} \cdot L^t( \bold{x}) = L(\bold{k}) \cdot \bold{x}$, where $\bold{k} \in \mathbb{Z}^d$  and $\bold{x} \in \mathbb{T}^e$. Let $\nu'$ be the image of $\nu$ under $L^t$. Then $$ \hat{\nu}'(\bold{d}) = \hat{\nu}(L(\bold{d})) \quad \text{and} \quad \nu'\{ (0, \dots, 0) \} \geq \nu\{ (0, \dots, 0) \}.$$
Then
 $$\nu\{(0,\dots,0)\} \leq \nu'\{(0,\dots, 0)\} \leq \limsup_{\bold{d} \in D, |\bold{d}| \rightarrow \infty} |\hat{\nu}'(\bold{d})| 
= \limsup_{\bold{k} \in L(D), |\bold{k}| \rightarrow \infty} |\hat{\nu}(\bold{k})|.\qedhere$$
\end{proof}

We will utilize the following lemmas to show that $D_1$ and $D_{-1}$ are nice $FC^+$ sets. 
\begin{Lemma}
\label{subsequence}
Let $\xi_1(x), \dots , \xi_k(x) \in {\bf H}$ such that $\sum_{i=1}^k b_i \xi_i(x) \in {\bf H}$ for any $(b_1, \dots, b_k) \in \mathbb{Z}^k \backslash \{(0, 0, \dots, 0)\}$.
Let $q \in \mathbb{N}$ and let $t$ be an integer with $0 \leq t \leq q-1$ and $(t,q) =1$. 
Then $(\xi_1(p), \dots, \xi_k(p) )$ is u.d. $\bmod \, 1$ along $p \in (t+ q\mathbb{Z}) \cap \mathcal{P}$.
\end{Lemma}

\begin{proof}
For $(a_1, a_2, \dots, a_k) \ne (0,0, \dots, 0)$ in $\mathbb{Z}^k$, let $f(x) = \sum_{i=1}^k a_i  \xi_i (x)$. Define $A_N = \{ p \leq N:  p \equiv t \bmod q \}$.
Then
\begin{align*}
\frac{1}{|A_N|} \sum\limits_{\substack{ p \leq N \\ p \equiv t \,\,  \bmod q}} e\left( f(p) \right) 
&= \frac{1}{|A_N|}\sum_{p \leq N} e \left( f(p) \right) \frac{1}{q} \sum_{j=1}^q e \left( \frac{(p-t)j}{q} \right) \\
&= \frac{\pi(N)}{|A_N|} \frac{1}{q}  \sum_{j=1}^q  \frac{1}{\pi(N)} \sum_{p \leq N} e \left( f (p) + \frac{j}{q}(p-t) \right)
\end{align*}
and the result follows from the fact that $\lim\limits_{N \rightarrow \infty} \frac{|A_N|}{\pi(N)} = \frac{1}{\phi(q)}$ and $ ( f (p) + \frac{j}{q}(p-t) )_{p \in \mathcal{P}}$ is u.d. $\bmod \, 1$.
\end{proof}

\begin{Lemma}[Lemma 4.1 in \cite{BKMST}]
\label{vdc}
Let $D \subset \mathbb{Z}^k$.  For each $q \in \mathbb{N}$, define 
$$D^{(q!)} := \{ \bold{d} = (d_1, d_2, \dots, d_k) \in E : q! \,\, \textrm{divides} \,\, d_i \,\, \textrm{for all} \,\,\, 1 \leq i \leq k \}.$$
Suppose that, for every $q$, there exists a sequence $(\bold{d}^{q,n})_{n \in \mathbb{N}}$ in $D^{(q!)}$ such that 
\begin{enumerate}[(i)]
\item $(|\bold{d}^{q,n}|)_{n \in \mathbb{N}}$ is non-decreasing and 
\item for any $\bold{x}=(x_1, \cdots , x_k) \in \mathbb{R}^k$, if one of $x_i$ is irrational, the sequence $(\bold{x} \cdot \bold{d}^{q,n})_{n \in \mathbb{N}}$ is uniformly distributed $\bmod \, 1$.
\end{enumerate}
Then $D$ is a nice $FC^+$ set.
\end{Lemma}

\begin{proof}[Proof of Theorem \ref{sarkozy type}]
We will prove that $D_1$ is a nice $FC^+$ set. 
The proof for $D_{-1}$ is analogous.
Let $$\bold{d}_n =\left(  p_n + 1, \cdots , (p_n + 1)^l, [\xi_1(p_n)], \cdots , [\xi_k (p_n )]  \right).$$
Let $D_1^{(q!)}:= \{ \bold{d} = (d_1, d_2, \dots, d_{l+k}) \in D_1 : q! \,\, \textrm{divides} \,\, d_i \,\, \textrm{for} \,\, 1 \leq i \leq l+k \}$. 

Let us first show that $D_1^{(q!)}$ has positive relative density in $D_1$. 
Consider the partition $\mathcal{P} = \bigcup\limits_{(t,q!) =1}( (t+ q!\mathbb{Z}) \bigcap \mathcal{P})$. 
The relative density of $(t+ q!\mathbb{Z}) \bigcap \mathcal{P}$ in $\mathcal{P}$ is $\frac{1}{\phi(q!)}$.
Now, if $p \in (t + q! \mathbb{Z}) \bigcap \mathcal{P}$, the pair of conditions 
$$ q! | (t + 1)^{c_i} \, (1 \leq i \leq l) \,\, \text{and} \,\, 0 \leq \left\{ \frac{\xi_i (p )}{q!} \right\} < \frac{1}{q!} \, (1 \leq i \leq k)$$
is equivalent to $\left(  P_1 (p + 1), \cdots , P_l (p + 1), [\xi_1 (p)], \dots , [\xi_k (p)]  \right) \in D_1^{(q!)}.$ 
Then $D_1^{(q!)}$ has positive relative density in $D_1$ since $\left( \frac{ \xi_1 (p)}{q!}, \dots ,  \frac{\xi_k(p)}{q!} \right)$ is uniformly distributed $\bmod \, 1$ in $\mathbb{T}^k$ 
along the increasing sequence of primes $p \in t + q! \mathbb{Z}$.

Now let $\bold{x}=(x_1, x_2, \dots, x_{l+k})$, where one of $x_i$ is irrational.
We need to prove that for any non-zero integer $m$,
\begin{equation}
\label{thm4.3}
\lim_{N \rightarrow \infty} \frac{1}{ | \{ n \leq N : \bold{d}_n \in D_1^{(q!)} \} | } \sum_{n \leq N, \bold{d}_n \in D_1^{(q!)} } e (m (\bold{d}_{n} \cdot \bold{x})) =0.
\end{equation}
Then, using Lemma \ref{Mo}, 
\begin{align*}
 &\frac{1}{ | \{ n \leq N : \bold{d}_n \in D_1^{(q!)} \} | } \sum_{n \leq N, \bold{d}_n \in D_1^{(q!)} } e (m (\bold{d}_{n} \cdot \bold{x}))  \\
&= \frac{1}{ | \{ n \leq N : \bold{d}_n \in D_1^{(q!)} \} | } \sum_{n \leq N  } e (m (\bold{d}_{n} \cdot \bold{x})) \frac{1}{(q!)^{l+k}} \sum_{j_1 = 1}^{q!} \cdots \sum_{j_{l+k} = 1}^{q!} e \left(\bold{d}_n \cdot \left(\frac{j_1}{q!}, \cdots , \frac{j_{l+k}}{q!} \right) \right) \\
&= \frac{N}{ | \{ n \leq N : \bold{d}_n \in D_1^{(q!)} \} | } \frac{1}{(q!)^{l+k}} \sum_{j_1 = 1}^{q!} \cdots \sum_{j_{l+k} = 1}^{q!} \frac{1}{N} \sum_{n \leq N}e \left( \bold{d}_n \cdot (m \bold{x} +\left(\frac{j_1}{q!}, \cdots , \frac{j_{l+k}}{q!} \right) \right).
\end{align*}
Then \eqref{thm4.3} holds from Proposition \ref{ud lemma}. 
\end{proof}

\begin{Corollary}
\label{cor nice}
If $E \subset \mathbb{Z}^{m}$ with ${d^*}(E) > 0$, then for any $\epsilon > 0$
\begin{equation*}
 R(E, \epsilon) := \{ \bold{d} \in \bold{D}_i : d^*(E \cap E - \bold{d} ) \geq d^*(E)^2 - \epsilon \} \,\, 
\end{equation*}
 is infinite.
\end{Corollary}
 
\subsection{Uniform distribution and sets of recurrence}\mbox{}

Let $\bold{D}_i = L(D_i)$ $(i = \pm 1)$ as in the previous section, that is, 
\begin{equation*}
%\label{Dset}
\begin{split}
 D_{-1} &= \{ \left(  p-1, \dots , (p-1)^l, [\xi_1(p)], \dots , [\xi_k (p)]  \right) | \, p \in\mathcal{P} \}, \\
 D_{1} &= \{ \left(  p+1, \dots , (p+1)^l, [\xi_1(p)], \dots , [\xi_k (p)]  \right) | \, p \in\mathcal{P} \},
\end{split}
\end{equation*}
where 
\begin{itemize}
\item $L: \mathbb{Z}^{l+k} \rightarrow \mathbb{Z}^m$ is a non-zero linear map.
%\item $P_1(x), \dots, P_l(x) \in \mathbb{Z}[x]$ are polynomials with $P_i(0) =0$ $(1 \leq i \leq l)$ such that $\sum_{i=1}^l a_i P_i(x) = 0$ implies $a_1= \cdots a_l=0$. 
\item $\xi_1 (x), \dots, \xi_k(x) \in {\bf H}$  such that $\sum_{i=1}^k b_i \xi_i(x) \in {\bf H}$ for any $(b_1, \dots, b_k) \in \mathbb{Z}^k \backslash \{(0, 0, \dots, 0)\}$.
\end{itemize}

\begin{Theorem}
\label{recurrence}
Enumerate the elements of $\bold{D}_i$, $(i = \pm 1)$, as follows:
$$\bold{d}_{n,i} = L \left( (p_n + i ), \dots , (p_n + i )^l, [\xi_1(p_n )], \dots , [\xi_k(p_n )]  \right), \quad {n =1, 2, \dots} .$$
For each $r \in \mathbb{N}$, let $\bold{D}_i^{(r)} = \bold{D}_i \cap \bigoplus\limits_{j=1}^{m} r \mathbb{Z}$ 
and enumerate the elements of $\bold{D}_i^{(r)}$ by $(\bold{d}_{n,i}^{(r) })$ such that  $|\bold{d}_{n,i}^{(r)}|$ is non-decreasing.\footnote{Note that, for any $r$ in $\mathbb{N}$, the sequence $|\bold{d}_{n,i}^{(r)}|$ is eventually increasing.} 
Let $(T^{\bold{d}})_{\bold{d} \in \mathbb{Z}^{m}}$ be a measure preserving $\mathbb{Z}^{m}$-action on a probability space $(X, \mathcal{B}, \mu)$.
Then
\begin{enumerate}[(i)]

\item $\bold{D}_i$ is an ``averaging" set of recurrence:
  \begin{equation}
  \label{eqn3.3.9}
  \lim_{N \rightarrow \infty} \frac{1}{N} \sum_{n=1}^N \mu(A \cap T^{-\bold{d}_{n,i}} A) > 0.
  \end{equation} 
\item For any $A \in \mathcal{B}$ with $\mu(A) > 0$ and for any $\epsilon > 0$, there exists $r \in \mathbb{N}$ such that
  \begin{equation}
  \label{eqn3.3.7}
  \lim_{N \rightarrow \infty} \frac{1}{N} \sum_{n=1}^N \mu(A \cap T^{-\bold{d}_{n,i}^{(r)}} A) \geq  \mu(A)^2 - \epsilon.
  \end{equation} 
Moreover, $ \{ \bold{d} \in \bold{D}_i : \mu(A \cap T^{-\bold{d}}A ) \geq \mu^2(A) - \epsilon \}$ has positive lower relative density in $\bold{D}_i$. 
Hence, $\bold{D}_i$ is a set of nice recurrence.
\end{enumerate}
\end{Theorem}

\begin{proof}
Let us prove this for the case that $L$ is an identity map on $\mathbb{Z}^{l+k}$ and $i=1$ first. To simplify the notation we use $\bold{d}_n$ and $\bold{d}_{n}^{(r)}$ instead of $\bold{d}_{n,i}$ and $\bold{d}_{n,i}^{(r)}$.

For any $\mathbb{Z}^{l+k}$-action $T$, there are commuting invertible measure preserving transformations $T_1, \dots , T_{l+k}$ such that $T^{\bold{m}} = T_1^{m_1} \cdots T_{l+k}^{m_{l+k}}$ for $\bold{m} = (m_1, m_2 , \dots , m_{l+k})$.

First we will show that the limits in \eqref{eqn3.3.9} and \eqref{eqn3.3.7} exist. 
By Theorem \ref{BH}, there exists a measure $\nu$ on $\mathbb{T}^{k+l}$ such that
$$\mu(A \cap T^{-\bold{n}}A) = \int 1_A(x) \,\, T^{\bold{n}} 1_A(x) \, d \mu(x) = \int_{\mathbb{T}^{l+k}} e(\bold{n} \cdot \bold{\gamma}) \, d \nu (\bold{\gamma}).$$
Thus, it is sufficient to show that for every $\gamma$, 
\begin{equation}
\label{eq4.6}
\lim\limits_{N \rightarrow \infty}
\frac{1}{N} \sum\limits_{n=1}^N e(\bold{d}_n \cdot \bold{\gamma}) \,\,\,\, \textrm{and} \,\,\,\,
 \lim\limits_{N \rightarrow \infty} \frac{1}{N}
\sum\limits_{n=1}^N e(\bold{d}_n^{(r)} \cdot \bold{\gamma})
\end{equation} exist. 
By \eqref{mo}, denoting $A_N =\{ n \leq N: \bold{d}_n \in D_1^{(r)} \}$,
\begin{align*}
&\lim_{N \rightarrow \infty} \frac{1}{N} \sum_{n=1}^N e (\bold{d}_n^{(r)} \cdot \bold{\gamma})\\
&= \lim_{N \rightarrow \infty} \frac{1}{|A_N|} \sum_{n=1}^N e\left(\bold{d}_n \cdot \bold{\gamma}\right) \,\, \left( \prod_{i=1}^l  \frac{1}{r} \sum_{j_i=1}^r e\left(\frac{  (p_n + 1)^i j_i}{r}\right) \right) \left( \prod_{i=1}^k \frac{1}{r} \sum_{j_{l+i}=1}^r e\left(\frac{[\xi_i (p_n)] j_{l+i}}{r}\right) \right) \\
&= \lim_{N \rightarrow \infty} \frac{N}{|A_N|} \frac{1}{r^{k+l}} \sum_{j_1 = 1}^r \cdots \sum_{j_{l+k} = 1}^{r} \frac{1}{N}
\sum_{n=1}^N e \left(\bold{d}_n \cdot (\bold{\gamma} + \left(\frac{j_1}{r} +     \cdots + \frac{j_{l+k}}{r} \right) \right).
\end{align*}
Using the same argument as in the proof of Theorem \ref{sarkozy type}, 
we can show that relative density of $D_1^{(r)}$ in $D_1$ is positive, 
so we only need to show that $\lim\limits_{N \rightarrow \infty} \frac{1}{N} \sum\limits_{n=1}^N e(\bold{d}_n \cdot \bold{\gamma})$ exists for every $\gamma$.

From Proposition \ref{ud lemma}, if $\bold{\gamma} \notin \mathbb{Q}^{l+k}$, 
then $\lim\limits_{N \rightarrow \infty} \frac{1}{N}\sum\limits_{n=1}^N e(\bold{d}_n \cdot \bold{\gamma}) = 0.$
If $\bold{\gamma} = (\gamma_1, \gamma_2, \dots, \gamma_{l+k}) \in \mathbb{Q}^{l+k}$, 
then we can find a common denominator $q \in \mathbb{N}$ for $\gamma_1, \dots,\gamma_{l+k}$ such that $\gamma_i = \frac{a_i}{q}$ for each $i$. 
We claim that the following limit exists:
$$\lim\limits_{N \rightarrow \infty}\dfrac{1}{\pi(N)} \sum\limits_{\substack { p  \equiv t \, \, \bmod q \\  p \leq N}} e\left(\sum_{j=1}^k  [\xi_j(p)] \frac{a_{l+j}}{q}\right).$$
This follows from two observations:
\begin{enumerate}[(a)]
\item $\{ p \in \mathcal{P} : p \equiv t \,\, \bmod q \}$ has a density $\frac{1}{\phi(q)}$ in $\mathcal{P}$ for $(t,q) = 1$.
\item $([\xi_1(p)] , \dots , [\xi_k(p)] )$ is u.d.\ in $\mathbb{Z}_q^k$ along $p \in t + q \mathbb{Z}$ for $(t,q) = 1$ 
since $\left( \frac{\xi_1(p)}{q}, \dots , \frac{\xi_k(p)}{q} \right)$ is u.d.\ $\bmod \, 1$ in $\mathbb{T}^k$ along $p \in t + q \mathbb{Z}$ from Lemma \ref{subsequence}.
\end{enumerate}
Thus, we have
\begin{align*} 
&\lim\limits_{N \rightarrow \infty} \frac{1}{N} \sum_{n=1}^N e(\bold{d}_n \cdot\bold{\gamma})
= \lim_{N \rightarrow \infty} \frac{1}{\pi(N)} \sum\limits_{p \leq N} e\left(\sum_{i=1}^l P_i (p+1) \frac{a_i}{q} + \sum_{j=1}^k [\xi_j(p)] \frac{a_{l+j}}{q}\right) \\
&\quad=\lim\limits_{N \rightarrow \infty}  \frac{1}{\pi(N)} \sum\limits_{\substack{ (t, q) =1 \\ 0 \leq t \leq q-1}} \sum\limits_{\substack{ p \equiv t \,\, \bmod q \\  p \leq N}} e\left(\sum_{i=1}^l P_i (p+1) \frac{a_i}{q} + \sum_{j=1}^k [\xi_j(p)] \frac{a_{l+j}}{q}\right) \\
&\quad= \sum\limits_{\substack{ (t, q) =1 \\ 0 \leq t \leq q-1}} e\left(\sum_{i=1}^l P_i (t+1) \frac{a_i}{q}\right) \lim\limits_{N \rightarrow \infty} \frac{1}{\pi(N)} \sum\limits_{\substack{ p \equiv t \,\, \bmod q \\  p \leq N}} e\left(\sum_{j=1}^k [\xi_j(p)] \frac{a_{l+j}}{q}\right). 
\end{align*}

Now let us show (ii). 
Consider the following Hilbert space splitting for $L^2(X) = \mathcal{H} = \mathcal{H}_{rat} \oplus \mathcal{H}_{tot}$ for $T_1, \dots, T_{l+k}$, where
$$\mathcal{H}_{rat} = \overline{ \{ f \in \mathcal{H} :  \textrm{there exists non-zero} \,\, (m_1, m_2, \dots, m_{l+k}) \in \mathbb{Z}^{l+k}, \, T_i^{m_i} f = f \,\, \textrm{for all} \,\, i  \} }, $$
\begin{eqnarray*}
\mathcal{H}_{tot} = \{ f \in \mathcal{H} &:&  \,\, \textrm{for any non-zero} \,\, (m_1, m_2,  \dots, m_{l+k}) \in \mathbb{Z}^{l+k} \\
&&\lim_{N_1, \cdots, N_k \rightarrow \infty} \left|\left|\frac{1}{N_1 \cdots N_{l+k}} \sum_{n_1=0}^{N_1-1} \cdots  \sum_{n_k=0}^{N_k-1} T_1^{m_1 n_1} \cdots T_{l+k}^{m_{l+k} n_{l+k}} f \right|\right|_{\mathcal{H}} = 0 \}.
\end{eqnarray*}
Let $1_A = f + g$, where $f \in \mathcal{H}_{rat} $ and $g \in \mathcal{H}_{tot}$.
Note that $\mathcal{H}_{rat} = \overline{\bigcup_{q=1}^{\infty} \mathcal{H}_q}$, where $\mathcal{H}_q = \{  f : T_i^{q!} f = f \,\, \textrm{for} \,\, i=1,2, \dots, l+k \}$.

For any $\epsilon >0$, there exists $\bold{a} = (a_1, \cdots , a_{l+k}) \in \mathbb{Z}^{l+k} $ and $f_{\bold{a}} \in \mathcal{H}_{rat} $ 
such that $T^{\bold{a}} f_{\bold{a}} = f_{\bold{a}}$, $|| f_{\bold{a}} - f || < \epsilon/2 $ and $\int f_{\bold{a}} \, d \mu = \mu(A)$.
Choose $r$ such that $a_i | r $ for all $i$. %Note that the set of $\{ \bold{d}_n^{(r)} \}$ has positive relative density in $D_1$ .
Consider 
$$ \frac{1}{N} \sum_{n=1}^N \mu(A \cap T^{-\bold{d}_n^{(r)}}A) =   
\frac{1}{N} \sum_{n=1}^N  \int f(x) \, T^{\bold{d}_n^{(r)}} f(x) \, d \mu(x)   +   \frac{1}{N} \sum_{n=1}^N  \int g(x) \,T^{\bold{d}_n^{(r)}} g(x) \, d \mu(x). $$
For $f \in \mathcal{H}_{rat}$,
\begin{eqnarray*}
\int f(x) \, T^{\bold{d}_n^{(r)}} f(x) \, d \mu(x) &=& \langle f_{\bold{a}}, f_{\bold{a}} \rangle + \langle f_{\bold{a}}, T^{\bold{d}_n^{(r)}} (f- f_{\bold{a}}) \rangle + \langle f-f_{\bold{a}}, T^{\bold{d}_n^{(r)}}f \rangle \\
& \geq& \mu^2(A) - \epsilon, 
\end{eqnarray*}
Also note that $(\bold{d}_n^{(r)}\cdot \bold{\gamma)}$ is u.d.\ $\bmod \, 1$ for $\bold{\gamma} \notin (\mathbb{Q}/\mathbb{Z})^{l+k}$. 
Hence,
$$\frac{1}{N} \sum_{n=1}^N  \int g(x) \,T^{\bold{d}_n^{(r)}} g(x) \, d \mu(x) = \int \frac{1}{N} \sum_{n=1}^N e(\bold{d}_n^{(r)}\cdot \bold{\gamma} ) \, d \nu(\bold{\gamma}) \rightarrow 0,$$
since $\nu(\mathbb{Q}/\mathbb{Z})^{l+k} = 0$ because $g \in \mathcal{H}_{tot}$.
Then, $$\frac{1}{N} \sum_{n=1}^N \mu(A \cap T^{-\bold{d}_n^{(r)}}A) \geq \mu(A)^2 - \epsilon,$$
so $ \{ \bold{d} \in D_1 : \mu(A \cap T^{-\bold{d}}A ) \geq \mu^2(A) - \epsilon \}$ has positive lower relative density in $D_1$.
 
Now it remains to show (i). Choose $\epsilon$ so small that $\mu^2(A) - \epsilon \geq \mu^2(A)/2$. 
Since $(\bold{d}_n^{(r)})$ has positive relative density, say $\delta$, we have
$$\lim_{N \rightarrow \infty} \frac{1}{N}\sum_{n=1}^N \mu(A \cap T^{-\bold{d}_n}A) 
\geq \delta \lim_{N \rightarrow \infty} \frac{1}{N}\sum_{n=1}^N \mu(A \cap T^{-\bold{d}_n^{(r)}}A) 
\geq \frac{\delta}{2} \mu^2(A). $$

The proof for $i = -1$ is completely analogous. 
By the argument in the proof of Theorem \ref{A} we can prove the theorem for any $\bold{D}_i= L(D_i)$.
\end{proof}

\begin{Corollary}
\label{semi ergodic cor}
If $E \subset \mathbb{Z}^{m}$ with ${d^*}(E) > 0$, then for any $\epsilon > 0$,
\begin{equation*}
\{ \bold{d} \in \bold{D}_i : d^*(E \cap E - \bold{d} ) \geq d^*(E)^2 - \epsilon \} 
 \end{equation*}
has positive lower relative density in $\bold{D}_i$ for $i = \pm 1$.
Furthermore,
$$\liminf_{N \rightarrow \infty} \frac{\left| \{ p \leq N : L \left(  p - 1, \dots , (p - 1)^l, [\xi_1(p)], \dots , [\xi_k(p)]  \right) \in E - E \} \right| }{\pi(N)} > 0.$$
$$\liminf_{N \rightarrow \infty} \frac{\left| \{ p \leq N: L \left(  p + 1, \dots , (p + 1)^l, [\xi_1(p)], \dots , [\xi_k(p)]  \right) \in E - E \} \right| }{\pi(N)} > 0.$$ 
\end{Corollary}

\begin{Remark}
\label{end}
Let $P_1(x), \dots, P_l(x) \in \mathbb{Z}[x]$ with $P_i(0)=0$ for all $1 \leq i \leq l$ and let  $\xi_1 (x), \dots, \xi_k(x) \in {\bf H}$  such that $\sum_{i=1}^k b_i \xi_i(x) \in {\bf H}$ for any $(b_1, \dots, b_k) \in \mathbb{Z}^k \backslash \{(0, 0, \dots, 0)\}$.
By choosing a linear map $L$ appropriately, the sets $D_{-1}$ and $D_1$ in Theorem \ref{recurrence} and Corollary \ref{semi ergodic cor} become
\begin{equation*}
\label{disply:polynomial}
\begin{split}
  \{ \left( P_1( p-1), P_2 (p-1), \dots , P_l(p-1), [\xi_1(p)], \dots , [\xi_k (p)]  \right) | \, p \in\mathcal{P} \}, \\
 \{ \left(  P_1(p+1), P_2(p+1), \dots , P_l(p+1), [\xi_1(p)], \dots , [\xi_k (p)]  \right) | \, p \in\mathcal{P} \}
\end{split}
\end{equation*}
respectively.
 \end{Remark}

%%%%%%%%%%%%%%%%%%%%%%%%%%%%%%%%%%%%%%%%%%%
\renewcommand{\abstractname}{Acknowledgments}
\begin{abstract}
The authors would like to thank Donald Robertson for careful reading of the manuscript and many useful comments.  
\end{abstract}

\end{document}